\documentclass{article}
\title{Thresholds for contagious sets in random graphs}
\author{Omer Angel \and Brett Kolesnik}
\date{}

\AtEndDocument{
  \bigskip
  \small
  \textsc{Omer Angel, Brett Kolesnik} \par
  \textsc{Department of Mathematics, University of British Columbia} \par
  \textit{Email:} \texttt{\{angel,bkolesnik\}@math.ubc.ca},
}

\usepackage[T1]{fontenc}
\usepackage{lmodern,microtype,hyperref,graphicx}
\usepackage{amsthm,amsmath,amssymb}
\usepackage{enumitem}

\theoremstyle{plain}

\newtheorem{thm}{Theorem}[section]
\newtheorem*{thm*}{Theorem}
\newtheorem{lemma}[thm]{Lemma}
\newtheorem{lem}[thm]{Lemma}
\newtheorem{claim}[thm]{Claim}
\newtheorem{subclaim}[thm]{Sub-claim}

\newtheorem{prop}[thm]{Proposition}

\theoremstyle{definition}
\newtheorem{defn}[thm]{Definition}

\theoremstyle{remark}

\numberwithin{equation}{section}

\newcommand{\eps}{\varepsilon}
\renewcommand{\P}{{\mathbb P}}
\newcommand{\N}{{\mathbb N}}

\renewcommand{\l}{\langle}
\renewcommand{\r}{\rangle}
\newcommand{\G}{{\cal G}}
\newcommand{\Gnp}{{\cal G}_{n,p}}

\newcommand{\NN}{{\mathbb N}}

\DeclareMathOperator{\Poi}{Poi}

\begin{document}

\maketitle

\begin{abstract}
  For fixed $r\geq 2$, we consider bootstrap percolation with threshold
  $r$ on the Erd{\H{o}}s-R{\'e}nyi graph ${\cal G}_{n,p}$.  We identify
  a threshold for $p$ above which there is with high probability a set
  of size $r$ which can infect the entire graph.  This improves a result
  of Feige, Krivelevich and Reichman, which gives bounds for this
  threshold, up to multiplicative constants.

  As an application of our results, we also obtain an upper bound for
  the threshold for $K_4$-bootstrap percolation on ${\cal G}_{n,p}$,
  as studied by Balogh, Bollob{\'a}s and Morris.  We conjecture 
  that our
  bound is asymptotically sharp.

  These thresholds are closely related to the survival probabilities of
  certain time-varying branching processes, and we derive asymptotic
  formulae for these survival probabilities which are of interest in
  their own right.
\end{abstract}

\section{Introduction}

\subsection{Bootstrap percolation}

The \emph{$r$-bootstrap percolation process} on a graph $G=(V,E)$
evolves as follows.  Initially, some set $V_0 \subset V$ is
\emph{infected}.  Subsequently, any vertex that has at least $r$
infected neighbours becomes infected, and remains infected.  Formally
the process is defined by
\[
  V_{t+1} = V_t \cup \Big\{v : |N(v) \cap V_t| \geq r \Big\},
\]
where $N(v)$ is the set of neighbours of a vertex $v$.  The sets $V_t$
are increasing, and so converge to some set $V_\infty$ of eventually
infected vertices.  We denote the infected set by
$\l V_0,G \r_r = V_\infty$.  A \emph{contagious set} for $G$ is a set
$I\subset V$ such that if we put $V_0=I$ then we have that
$\l I,G \r_r = V$, that is, the infection of $I$ results in the
infection of all vertices of $G$.

Bootstrap percolation was introduced by Chalupa, Leath and
Reich~\cite{CLR79}, in the context of statistical physics, for the study
of disordered magnetic systems.  Since then it has been applied
diversely in physics, and in other areas, including computer science,
neural networks, and sociology, see \cite{A91,AL03,A10, GLBD05,DR09,
  FLP10,FLLPS04,FSS02,F89, GST11,G78, FWGH02,M11, TE09, vE87,W02,W86} and
further references therein.

Special cases of $r$-bootstrap percolation have been analyzed
extensively on finite grids and infinite lattices, see for instance
\cite{AL88,BBDCM12,BBM09b,BBM10,BP97, CM02,CC99,GHM12,H03,HLR04,S92}
(and references therein).  Other special graphs of interest have also
been studied, including hypercubes and trees, see
\cite{BB06,BBM09a,BPP06,FS08}.  Recent work has focused on the case of
random graphs, see for example \cite{A10,AF12,BP07,J09}, and in
particular, on the Erd\H{o}s-R\'enyi random graph $\Gnp$.  See
\cite{JLTV12,V07} (and \cite{BB05,BB09,ST85} for related results).

The main questions of interest in this field revolve around the size of
the eventual infected set $V_\infty$.  In most works, the object of
study is the probability that a \emph{random} initial set is contagious,
and its dependence on the size of $V_0$.  For example, in
\cite[Theorem~3.1]{JLTV12}, the critical size for a random contagious
set in $\Gnp$ is identified for all $r\ge2$ and $p$ in a range depending
on $r$.

More recently, and in contrast with the above results, Feige,
Krivelevich and Reichman~\cite{FKR16} study the \emph{minimal size} of a
contagious set in $\Gnp$.  We define a \emph{minimal} contagious set to
be a contagious set of size $r$.  This is the minimal possible size of a
contagious set.  We call a graph \emph{susceptible} 
(or say that it \emph{$r$-percolates}) if it contains such
a minimal contagious set.  A main result, \cite[Theorem~1.2]{FKR16},
gives the approximate threshold for $p$ above which $\Gnp$ is likely to
be susceptible.  Our main result identifies sharp thresholds for
the susceptibility of $\Gnp$, for all $r\ge2$.

Let $p_c(n,r)$ denote the infimum over $p>0$ so that $\Gnp$
is susceptible with probability at least $1/2$. 

\begin{thm}\label{T_main}
  Fix $r\ge2$ and $\alpha>0$.  Let 
  \[
    p=p(n)=\left(\frac{\alpha}{n\log^{r-1}n}\right)^{1/r}
  \]
  and denote
  \[
    \alpha_r=(r-1)!\left(\frac{r-1}{r}\right)^{2(r-1)}.
  \]
  If $\alpha>\alpha_r$, then with high probability $\Gnp$ is
  susceptible.
  If $\alpha<\alpha_r$, then there exists
  $\beta=\beta(\alpha,r)$ so that for $G=\Gnp$, with high probability
  for every $I$ of size $r$ we have $|\l I,G \r_r| \leq \beta\log n$.
  In particular, as $n\to\infty$, 
  \[
  p_c(n,r)=\left(\frac{\alpha_r}{n\log^{r-1}n}\right)^{1/r}(1+o(1)).
  \]
\end{thm}

Thus $r$-bootstrap percolation undergoes a sharp transition.  For small
$p$ sets of size $r$ infect at most $O(\log n)$ vertices, whereas for
larger $p$ there are minimal contagious sets.  

We remark that for $\alpha<\alpha_r$, with high probability $\Gnp$
has susceptible subgraphs of size $\Theta(\log n)$.  Moreover, our
methods identify the largest $\beta$ so that there are susceptible
subgraphs of size $\beta\log n$ (see Proposition~\ref{P_SC} below).

\subsection{Graph bootstrap percolation and seeds}\label{S_seeds}

Let $H$ be some finite graph.  Following Bollob\'as~\cite{B67},
\emph{$H$-bootstrap percolation} is a rule for adding edges to a graph
$G$.  Eventually no further edges can be added, and the process
terminates.  An edge is added whenever its addition creates a copy of
$H$ within $G$.  Informally, the process completes all copies of $H$
that are missing a single edge.  Formally, we let $G_0=G$, and $G_{i+1}$
is $G_i$ together with every edge whose addition creates a subgraph
which is isomorphic to $H$.  Note that these are not necessarily induced
subgraphs, so having more edges in $G$ can only increase the final
result.  The vertex set is fixed, and no vertices play any special role.

For a finite graph $G$, this procedure terminates once
$G_{\tau+1}=G_{\tau}$, for some $\tau=\tau(G)$.  We denote the resulting
graph $G_\tau$ by $\l G\r_H$.  If $\l G\r_H$ is the complete graph on
the vertex set $V$, the graph $G$ is said to \emph{$H$-percolate} (or 
that it is $H$-percolating).
The case $H=K_4$ is the minimal case of interest.  Indeed, \emph{all}
graphs $K_2$-percolate, and a graph $K_3$-percolates if and only if it
is connected.  Hence by a classical result of Erd{\H{o}}s and R{\'e}nyi
\cite{ER59}, $\Gnp$ will $K_3$-percolate precisely for $p> n^{-1}\log{n}
+ \Theta(n^{-1})$. 

The main focus of \cite{BBM12} is $H$-bootstrap percolation in the case
that $G=\Gnp$ and $H=K_k$, for some $k\ge4$.
The critical thresholds are defined as
\[
  p_c(n,H) = \inf\left\{p>0 : \P(\l\Gnp\r_H = K_n) \ge 1/2 \right\}.
\]
It is expected that this property has a sharp threshold for
$H=K_k$ for any $k$, in the sense that for some $p_c=p_c(k)$ we have
that $\Gnp$ is $K_k$-percolating with high probability for
$p>(1+\delta)p_c$ and is $K_k$-percolating with probability tending to
$0$ for $p=(1-\delta)p_c$.  

Some bounds on $p_c(n,K_k)$, $k\ge4$, are obtained in \cite{BBM12}.  One
of the main results of \cite{BBM12} is that
$p_c(n,K_4) = \Theta(1/\sqrt{n\log{n}})$.  We improve the upper bound on
$p_c(n,K_4)$ given in \cite{BBM12}.  We believe that the bound below is
asymptotically sharp.

\begin{thm}\label{T_K4}
  Let $p=\sqrt{\alpha/(n\log{n})}$.  If $\alpha>1/3$
  then $\Gnp$ is $K_4$-percolating with high probability.  In particular
  as $n\to\infty$, we have that
  \[
    p_c(n,K_4) \le \frac{1+o(1)}{\sqrt{3n\log{n}}}.
  \]
\end{thm}

One way for a graph $G$ to $K_{r+2}$-percolate is if there is some
ordering of the vertices so that vertices $1\dots,r$ form a clique, and
every later vertex is connected to at least $r$ of the previous vertices
according to the order.  In this case we call the clique formed by the
first $r$ vertices a \emph{seed} for $G$.  When $r=2$, the seed is a
clique of size $2$, so we call it a \emph{seed edge}.

\begin{lem}\label{L:seed_suffice}
  Fix $r\geq 2$. If $G$ has a seed for $K_{r+2}$-bootstrap percolation,
  then $\l G \r_{K_{r+2}} = K_n$.
\end{lem}

\begin{proof}
  We prove by induction that for $k\geq r$ the subgraph induced by the
  first $k$ vertices percolates.  For $k=r$, the definition of
  a seed implies that the subgraph is complete.  Given that the first
  $k-1$ vertices span a percolating graph, some number of steps will add
  all edges among them.  Finally, vertex $k$ has $r$ neighbours among
  these, and so every edge between vertex $k$ and a previous vertex can
  also be added by $K_{r+2}$-bootstrap percolation.
\end{proof}

In light of this, Theorem~\ref{T_K4} above is a direct corollary of the
following result.

\begin{thm}\label{T_SP}
  Let $p=\sqrt{\alpha/(n\log{n})}$.  As $n\to\infty$, the probability
  that $\Gnp$ has a seed edge tends to $1$ if $\alpha > 1/3$ and tends
  to $0$ if $\alpha<1/3$.
\end{thm}

The case of $K_4$-bootstrap percolation, corresponding to $r=2$, appears
to be special:
We conjecture that existence of a seed edge is the easiest way for a
graph to $K_4$-percolate, and consequently that the inequality in
Theorem~\ref{T_K4} can be made an equality, identifying the asymptotic
threshold for $K_4$-percolation.  This is similar to other situations
where a threshold of interest on $\Gnp$ coincides with that of a more
fundamental event.  For instance, 
with high probability,  
$\Gnp$ is connected if and only if it
has no isolated vertices (see \cite{ER59}); $\Gnp$ contains a
Hamiltonian cycle if and only if the minimal degree is at least $2$
(Koml\'os and Szemer\'edi~\cite{KS83}).

Essentially, if $G$ $K_4$-percolates, then either there is a seed edge,
or some other small structure that serves as a seed (i.e.,
$K_4$-percolates and exhausts $G$ by adding doubly connected vertices),
or else, there are at least two large structures within $G$ that
$K_4$-percolate independently.  Since $p_c\to 0$, having multiple large
percolating structures within $G$ is less likely.

For $r>2$, having a seed is no longer the easiest way for a graph to
$K_4$-percolate.  Indeed, by \cite{BBM12}, the critical probability for
$K_{r+2}$-bootstrap percolation is $n^{-(2r)/(r^2+3r-2)}$ up to
poly-logarithmic factors (note that $r$ in $\cite{BBM12}$ is $r+2$ here).
The threshold for having a seed is of order $n^{-1/r} (\log n)^{1/r-1}$,
which is much larger (see Theorem~\ref{T_pcHr}).

\subsection{A non-homogeneous branching process}
\label{S_intro:BP}

Given an edge $e=(x_0,x_1)$, we can explore the graph to determine if it
is a seed edge.  The number of vertices that are connected to both of
its endpoints is roughly Poisson with mean $np^2$.  In our context, the
interesting $p$ are $o(n^{-1/2})$, and therefore the number of such
vertices has small mean, which we denote by $\eps=np^2$.  If there are
any such vertices, denote them $x_2,\dots$. We then seek vertices
connected to $x_2$ and at least one of $x_0,x_1$.  The number of such
vertices is roughly $\Poi(2\eps)$.  Indeed, the number of vertices
connected to the $k$th vertex and at least one of the previous vertices is (approximately)
$\Poi(k\eps)$.

This leads us to the case $r=2$ of the following non-homogeneous branching
process defined by parameters $r\in\N$ and $\eps>0$.  The process starts
with a single individual.  The first $r-2$ individuals have precisely
one child each.  For $n\geq r-1$, the $n$th individual has a Poisson number
of children with mean $\binom{n}{r-1} \eps$,
where here $\eps=np^r$.  Thus for $r=2$ the $n$th
individual has a mean of $n\eps$ children.  The process may die out
(e.g., if individual $r-1$ has no children).  However, if the process
survives long enough the mean number of children exceeds one and the
process becomes super-critical.  Thus the probability of survival is
strictly between $0$ and $1$.  Formally, this may be defined in terms of
independent random variables $Z_n = \Poi\left(\binom{n}{r-1}\eps\right)$
by $X_t = \sum_{n=r-1}^{t} Z_n - 1$.  Survival is the event
$\{X_t \geq 0, \forall t\}$.

\begin{thm}\label{T_BP}
  As $\eps\to0$, we have that
  \[
    \P(X_t>0, \forall t) =
    \exp \left[ -\frac{(r-1)^2}{r} k_r (1+o(1)) \right]
  \]
  where
  \[
    k_r = k_r(\eps) = \left(\frac{(r-1)!}{\eps}\right)^{1/(r-1)}.
  \]
\end{thm}

Note that $\eps{k_r\choose r-1}\approx1$.  Hence $k_r$ is roughly the
time at which the process becomes super-critical.

\subsection{Outline of the proof}

In Section~\ref{S_pcLB}, we obtain a recurrence 
\eqref{E_mrec}
for the number of graphs which $r$-percolate
with the minimal number of edges.
Using this, we estimate the 
asymptotics of such graphs, and thereby  
identify a quantity $\beta_*(\alpha)$, 
so that for $\alpha<\alpha_r$ 
(and $p$ as in Theorem~\ref{T_main}), 
with high probability no $r$-percolation
on $\Gnp$ grows to size $\beta\log{n}$, 
for some $\beta \ge\beta_*(\alpha)+\delta$.
We put $\beta_r(\alpha)=k_r(np^r)$, where
$k_r$ is as in Section~\ref{S_intro:BP}.
Moreover, we find that $\beta_*(\alpha)=\beta_r(\alpha)$
if and only if $\alpha=\alpha_r$, 
suggesting that 
$\alpha_r$ is indeed the critical value of $\alpha$. 

In Section~\ref{S_pcUB}, we show by the 
second moment method that, if $\alpha>\alpha_r$, 
then $\Gnp$ $r$-percolates with high probability.
The main difficulty
is showing that contagious sets are sufficiently independent. 
Since vertices in a contagious
set need not be connected, 
it seems that 
perhaps a straightforward argument is not available. 
We instead study contagious sets which infect
\emph{triangle-free} subgraphs of $\Gnp$. 
Modifying the recurrence \eqref{E_mrec}, we
obtain a recursive lower bound for graphs which $r$-percolate
without using triangles, 
and find that this restriction
does not significantly effect the asymptotics. 
Using Mantel's theorem, 
we establish the approximate 
independence of correspondingly restricted 
$r$-percolations, which we call $\hat r$-percolations, 
with comparative ease.

A secondary obstacle is the need for 
a lower bound on the asymptotics of graphs which $\hat r$-percolate,
with a significant proportion of 
vertices in the \emph{top level} 
(i.e.,
vertices $v$ of a graph $G=(V,E)$ such that 
$v\in V_t\setminus V_{t-1}$ where $V_t=V$).
Such bounds are required to estimate the growth
of super-critical $\hat r$-percolations on $\Gnp$, 
which have grown
larger than the critical size $\beta_r(\alpha)\log{n}$. 
Using a lower bound 
for the overall number of graphs which $\hat r$-percolate, 
we obtain a lower bound on the number of such
graphs with $i=\Omega(k)$ vertices in the top level.
This estimate, 
together with the approximate independence result, 
is sufficient to show
that with high probability $\Gnp$ 
has subgraphs of size $\beta\log{n}$ 
which $r$-percolate,
for some $\beta\ge \beta_*(\alpha)+\delta$
(where for $\alpha>\alpha_r$, $\beta_r(\alpha)<\beta_*(\alpha)$).  

Finally, to conclude, we show 
by the first moment method that for any given $A>0$,  
with high probability
an $r$-percolation which survives to size 
$(\beta_*(\alpha)+\delta)\log{n}$ survives to size $A\log{n}$. 
Having established the existence of a subgraph of 
$\Gnp$ of size $A\log{n}$, for a 
sufficiently large value of $A$ (depending on 
the difference $\alpha-\alpha_r$), it is 
straightforward to show that 
with high probability
$\Gnp$ $r$-percolates.

\section{Lower bound for $p_c(n,r)$}\label{S_pcLB}

In this section, we prove the sub-critical case of Theorem~\ref{T_main}, by the
first moment method.  Throughout this section we fix some $r\geq 2$.
More precisely, we prove the following

\begin{prop}\label{P_SC}
  Let 
  \begin{align*}
    \alpha_r &= (r-1)!\left(\frac{r-1}{r}\right)^{2(r-1)}, &
    p=\vartheta_r(\alpha,n)&=\left(\frac{\alpha}{n\log^{r-1}n}\right)^{1/r}.
  \end{align*}
  Define $\beta_*(\alpha)$ to be the unique positive root of 
  \[
    r + \beta\log\left(\frac{\alpha\beta^{r-1}}{(r-1)!}\right)
    -\frac{\alpha\beta^r}{r!} 
    -\beta(r-2).
  \]
  For any $\alpha<\alpha_r$ and $\delta>0$, with high probability, for
  every $I\subset[n]$ of size $r$, we have that
  $|\l I,\Gnp \r_r| \leq (\beta_*(\alpha)+\delta)\log n$.
\end{prop}

The methods of Section~\ref{S_pcUB} can be used to show that 
with high probability 
there are
sets $I$ of size $r$ which infect $(\beta_*-\delta) \log n$ vertices.  For
$\alpha<\alpha_r$, we have the following upper bound 
\[
\beta_*(\alpha) \leq \left(\frac{(r-1)!}{\alpha}\right)^{1/(r-1)}.
\]
This is asymptotically optimal for $\alpha\sim \alpha_r$.

\subsection{Small susceptible  graphs}\label{S_m}

As noted in the introduction, a key idea is to study the number of
subgraphs of size $k = \Theta(\log n)$ which are susceptible
with the minimal number of edges.
If none exist, then there can be no contagious set in $G$.  Thus an
important step is developing estimates for the number of 
such susceptible
graphs of size $k$.  

For a graph $G$ and initial infected set $V_0$, recall that $V_t=V_t(V_0,G)$ is the
set of vertices infected up to and including step $t$.  We let
$\tau = \inf\{t : V_t=V_{t+1}\}$.  
We put $I_0=V_0$ and $I_t=V_t\setminus V_{t-1}$, for $t\ge1$. 
We refer to $I_t$ as the set of vertices
infected in \emph{level $i$}.
In particular, the \emph{top level} of $G$ is 
$I_\tau$.

For a graph $G$, we let $V(G)$ and $E(G)$ denote
its vertex and edge sets, and put $|G|=|V(G)|$. 

We call a graph \emph{minimally susceptible} if it is susceptible and
has exactly $r(|G|-r)$ edges.  If a graph $G$ is susceptible, it
has at least $r(|G|-r)$ edges, since each 
vertex in $I_t$, $t\ge1$, is connected to $r$ vertices
in $V_{t-1}$. 

\begin{defn}
  Let $m_r(k)$ denote the number of minimally susceptible graphs $G$
  with vertex set $[k]$ such that $[r]$ is a contagious set for $G$.
  Let $m_r(k,i)$ denote the number of such graphs with $i$ vertices
  infected in the top level  (so that $m_r(k)=\sum_{i=1}^{k-r}m_r(k,i)$).
\end{defn}

We note that $m_{r}(k,k-r)=1$, and claim that for $i<k-r$,
\begin{equation}\label{E_mrec}
  m_r(k,i) = {k-r\choose i} \sum_{j=1}^{k-r-i} a_r(k-i,j)^i m_r(k-i,j),
\end{equation}
where
\begin{equation}\label{E_a}
  a_r(x,y) = {x\choose r}-{x-y\choose r}.
\end{equation}

To see this, note that removing the top level from a minimally
susceptible graph $G$ of size $k$ leaves a minimally susceptible graph
$G'$ of size $k-i$.  If the top level of $G'$ has size $j$, then all
vertices in the top level of $G$ are connected to $r$ vertices of $G'$,
with at least one in the top level of $G'$.  Thus each vertex has
$a_r(k-i,j)$ options for the connections.  The $\binom{k-r}{i}$ term
accounts for the set of possible labels of the top level of $G$.

To study asymptotics of $m$ it is convenient to define
\begin{equation}\label{E_sig}
  \sigma_r(k,i) =
  \frac{m_r(k,i)}{(k-r)!}\left(\frac{(r-1)!}{k^{r-1}}\right)^k .
\end{equation}
Substituting this in \eqref{E_mrec} gives
\begin{equation}\label{E_sigrec}
\sigma_r(k,i) = \sum_{j=1}^{k-r-i} A_r(k,i,j) \sigma_r(k-i,j)
\quad \text{for $i< k-r$,}
\end{equation}
where 
\begin{equation}\label{E_A}
  A_r(k,i,j) = \frac{j^i}{i!} \left(\frac{k-i}{k}\right)^{(r-1)k}
  \left(\frac{(r-1)!}{(k-i)^{r-1}}\frac{a_r(k-i,j)}{j}\right)^i.
\end{equation}

We make the following observation.

\begin{lem}\label{L_A}
  Let $A_r(k,i,j)$ be as in \eqref{E_A} and put
  $A_r(i,j) = j^i e^{-(r-1)i}/i!$.  For any $i<k-r$ and $j\le k-r-i$, we
  have that $A_r(k,i,j)$ is increasing in $k$ and converges to $A_r(i,j)$.
\end{lem}

\begin{proof}
It is well known that for $m>0$ we have $(1-m/k)^k$ is increasing and tends
to $e^{-m}$.  Thus
\[
\frac{j^i}{i!}\left(\frac{k-i}{k}\right)^{(r-1)k} \to A_r(i,j).
\]
The lemma follows by \eqref{E_A} and the following claim, a formula which
will also be of later use.

\begin{claim}\label{C_a}
For all integers $x\ge r$ and $1\le y\le x-r$, we have that 
\[
\frac{(r-1)!}{x^{r-1}}\frac{a_r(x,y)}{y}
=
\frac{1}{y} \sum_{\ell=1}^{y}\left(\frac{x-\ell}{x}\right)^{r-1}.
\]
\end{claim}

\begin{proof}
For an integer $m\ge r$, let  
$(m)_r=m!/(m-r)!$ denote the $r$th falling factorial 
of the integer $m$. 
Since 
\[
(m)_r-(m-1)_r=r(m-1)^{r-1}.
\] 
it follows that 
\[
\frac{(r-1)!}{x^{r-1}}\frac{a_r(x,y)}{y}
=\frac{(x)_r-(x-y)_r}{ryx^{r-1}}
=\frac{1}{y} \sum_{\ell=1}^{y}\left(\frac{x-\ell}{x}\right)^{r-1}
\]
as required. 
\end{proof}

Since each term on the right of Claim~\ref{C_a} is increasing to $1$, the same
holds for their average. The proof is complete.
\end{proof}

\subsection{Upper bounds for susceptible graphs}
\label{S_mUB}

Our first task is to derive bounds on the number of 
minimally susceptible  graphs
of size $k$ with $i$ vertices in the top level.  
This relies on the
recurrence \eqref{E_mrec}.

\begin{lem}\label{L_mUB}
  Fix $r\ge2$.  For all $k>r$ and $i\le k-r$, we have that
  \[
    m_r(k,i) \le \frac{e^{-i-(r-2)k}}{\sqrt{i}}
    (k-r)! \left(\frac{k^{r-1}}{(r-1)!}\right)^k.
  \]
  Equivalently, $\sigma_r(k,i) \leq i^{-1/2} e^{-i - (r-2)k}$.
\end{lem}

\begin{proof}
  Since $m_r(k,k-r)=1$, it is straightforward to verify that the claim
  holds in the case that $i=k-r$. 

  For the remaining cases $i<k-r$, we prove the claim by induction on
  $k$.  Applying the inductive hypothesis to the right hand sum of
  \eqref{E_sigrec}, bounding $A_r(k,i,j)$ therein by $A_r(i,j)$
  using Lemma~\ref{L_A}, and
  extending the sum to all $j$ we have
  \[
    \sigma_r(k,i) \leq \sum_{j=1}^{\infty} A_r(i,j) j^{-1/2}
    e^{-j-(r-2)(k-i)}.
  \]
  Thus it suffices to prove that this sum is at most
  $i^{-1/2} e^{-i-(r-2)k}$.  Using the definition of $A_r(i,j)$ and
  cancelling the $e^{-(2-r)k}$ factors, we need the following
  
  \begin{claim}\label{C:induct}
    For any $i\geq 1$ we have
    \[
      \sum_{j=1}^{\infty} \frac{j^i e^{-i}}{i!} j^{-1/2} e^{-j} \leq
      i^{-1/2} e^{-i}.
    \]
  \end{claim}

  This is proved in Appendix~\ref{A:induct}.

We remark that Claim~\ref{C:induct} is fundamentally a pointwise bound on the
Perron eigenvector of the infinite operator $A_2$. (Other values of $r$ follow
since the influence of $r$ cancels out.)  This eigenvector decays
roughly as $e^{-i}$, but with some lower order fluctuations.  It appears
that the $\sqrt{i}$ correction can be replaced by various other slowly
growing functions of $i$.  However, Claim~\ref{C:induct} fails for certain
$i$ without the $\sqrt{j}$ term.
\end{proof}

\subsection{Susceptible subgraphs of $\Gnp$}
\label{S_Pki} 

With Lemma~\ref{L_mUB} at hand, we obtain upper bounds on the growth
probabilities of $r$-percolations on $\Gnp$.  

A set $I$ of size $r$ is
called \emph{$k$-contagious} in the graph $\Gnp$, if there is some $t$
so that $|V_t(I,\Gnp)| = k$, i.e., there is some time at which there are
exactly $k$ infected vertices.  The set $I$ is called $(k,i)$-contagious
if in addition the number of vertices infected at step $t$ is $i$, i.e.,
$|I_t(I,\Gnp)|=i$.  
Let $P_{r}(k,i)=P_{r}(p,k,i)$ denote the probability that a given
$I\subset [n]$, with $|I|=r$ is $(k,i)$-contagious.  Let
$P_r(k) = \sum_i P_r(k,i)$ denote the probability that such an $I$ is
$k$-contagious.  Finally, let $E_{r}(k,i)$ and $E_{r}(k)$ denote the
expected number of such subsets $I$.

We remark that $P_r(k)$ is not the same as the probability of
\emph{survival} to size $k$, which is given by
$\sum_{\ell\ge k}\sum_{i>\ell-k}P_{r}(\ell,i)$.

\begin{lem}\label{L_Pki}
  Let $\alpha>0$, and let $p=\vartheta_r(\alpha,n)$ (as defined in
  Proposition~\ref{P_SC}) and $\eps=np^r=\alpha/\log^{r-1}{n}$.  For
  $i\le k-r$ and $k\le n^{1/(r(r+1))}$, we have that
  \[
    P_{r}(k,i) \le (1+o(1)) 
    \frac{e^{-\eps{k-i\choose r}}\eps^{k-r}}{(k-r)!}m_{r}(k,i)
  \]
  where $o(1)$ depends on $n$, but not on $i,k$. 
\end{lem}

\begin{proof}
Let $I\subset[n]$, with $|I|=r$, be given, and put 
\[
\ell_{r}(k,i)=\frac{e^{-\eps{k-i\choose r}}\eps^{k-r}}{(k-r)!}m_{r}(k,i)
\]
so that the lemma states $P_r(k,i)\le(1+o(1))\ell_{r}(k,i)$.
  This follows 
  by a union bound:  
  If $I$ is $(k,i)$-contagious, 
then $I$ is a contagious set for a minimally susceptible subgraph  
$G\subset \Gnp$ (perhaps not induced) of size $k$ with $i$ vertices 
infected in the top level, 
and all vertices in $v\in V(G)^c$ are connected to at 
most $r-1$ vertices below the top level of $G$ (so that 
$V(G)=V_t(I,\Gnp)$, for some $t$).
There are $\binom{n}{k-r}$ choices for the vertices of $G$
and $m_r(k,i)$ choices for its edges.  
For any such $v$ and $G$, the probability
that $v$ is connected to $r$ vertices below the top level of $G$
is bounded from below by
\[
{k-i\choose r}p^r(1-p)^{k-i-r}
>
{k-i\choose r}p^r(1-p)^{k}.
\]
Hence
\[
P_{r}(k,i)<
{n\choose k-r}m_{r}(k,i)p^{r(k-r)}
\left(1-{k-i\choose r}p^r(1-p)^k\right)^{n-k}.
\]

By the inequalities ${n\choose k}\le n^k/k!$ and $1-x<e^{-x}$, 
it follows that 
\[
\log \frac{P_{r}(k,i)}{\ell_{r}(k,i)}
<
\eps{k-i\choose r}\left(1-(1-p)^k\left(1-\frac{k}{n}\right)\right).
\]
By the inequality $(1-x)^y\ge1-xy$, and since $k\le n^{1/(r(r+1))}$, 
the right hand side is bounded by 
\[
\eps k^{r+1}(p+(1-pk)/n)\le\eps n^{1/r}(p+1/n)\ll1
\]
as $n\to\infty$.
Hence $P_{r}(k,i)\le (1+o(1))\ell_{r}(k,i)$, 
as claimed. 
\end{proof}

As a corollary we get a bound for $E_r(k,i)$.

\begin{lem}\label{L_Eki}
  Let $\alpha,\beta_0>0$.  Put $p=\vartheta_r(\alpha,n)$.  For all
  $k=\beta\log{n}$ and $i=\gamma k$, such that $\beta\le \beta_0$, we
  have that
  \[
    E_{r}(k,i) \lesssim n^{\mu} \log^{r(r-1)}n
  \]
  where
  \begin{equation}\label{E_mu}
    \mu = \mu_r(\alpha,\beta,\gamma)
    =
    r+
    \beta\log\left(\frac{\alpha\beta^{r-1}}{(r-1)!}\right)
    -\frac{\alpha\beta^r}{r!}(1-\gamma)^r
    -\beta(r-2+\gamma).
  \end{equation}
  Here $\lesssim$ denotes inequality up to a constant depending on
  $\alpha,\beta_0$, but not on $\beta,\gamma$.
\end{lem}

\begin{proof}
Let $r\ge2$ and $\alpha,\beta_0>0$ be given. 
Put $\eps=np^r$. 
By 
Lemmas~\ref{L_mUB},\ref{L_Pki}, 
for all 
$k=\beta\log{n}$ and $i=\gamma k$, 
with $\beta\le\beta_0$, 
we have that 
\[ 
E_{r}(k,i)\le (1+o(1)){n\choose r}\left(\frac{\eps k^{r-1}}{(r-1)!}\right)^k
\eps^{-r}
e^{-i-(r-2)k-\eps{k-i\choose r}}
\lesssim 
n^{\mu}\log^{r(r-1)}n.
\]
The $\sqrt{i}$ term from Lemma~\ref{L_mUB} is safely dropped for this upper bound.  
\end{proof}

\subsection{Sub-critical bounds}
\label{S_sub-rperc}

In this section, we prove Proposition~\ref{P_SC}.

The case of $\gamma=0$ in Lemma~\ref{L_Eki} 
(corresponding to values of $i$ such that
$i/k\ll1$) is of particular importance for the growth of sub-critical
$r$-percolations.  For this reason, we introduce the notation
$\mu^*(\alpha,\beta) = \mu(\alpha,\beta,0)$.  
The next result in particular shows that 
$\beta_*(\alpha)$, as in Proposition~\ref{P_SC}, 
is well-defined. 

\begin{lem}\label{L_beta*}
  Let $\alpha>0$. 
  Let $\alpha_r$   be as in Proposition~\ref{P_SC}. 
  Put 
  \[
  \beta_r(\alpha)=\left(\frac{(r-1)!}{\alpha}\right)^{1/(r-1)}.
  \]
  \begin{enumerate}[nolistsep,label={(\roman*)}]
  \item The function $\mu_r^*(\alpha,\beta)$ is decreasing in $\beta$,
    with a unique zero at $\beta_*(\alpha)$.
  \item
    We have that 
    \[
      \mu_r^*(\alpha,\beta_r(\alpha)) = r-\beta_r(\alpha)\frac{(r-1)^2}{r}
    \]
    and hence 
    $\beta_*(\alpha)=\beta_r(\alpha)$ (resp.\ $>$ or $<$) if
    $\alpha=\alpha_r$ (resp.\ $>$ or $<$).
  \end{enumerate}
\end{lem}

The quantity $\beta_*(\alpha)$ also plays a 
crucial role in analyzing the 
growth of super-critical $r$-percolations on $\Gnp$, see
Section~\ref{S_term} below.

\begin{proof}
For the first claim, 
we note that by setting $\gamma=0$
in \eqref{E_mu} we obtain
\begin{equation}\label{E_mu*}
\mu_r^*(\alpha,\beta)
=
r+
\beta\log\left(\frac{\alpha\beta^{r-1}}{(r-1)!}\right)
-\frac{\alpha\beta^r}{r!}
-\beta(r-2).
\end{equation}

Therefore 
\[
\frac{\partial}{\partial\beta}\mu_r^*(\alpha,\beta)
=1
+\log\left(\frac{\alpha\beta^{r-1}}{(r-1)!}\right)
-\frac{\alpha\beta^{r-1}}{(r-1)!}.
\] 
Since 
$\alpha\beta_r(\alpha)^{r-1}/(r-1)!=1$,  
the above expression
is equal to 0 at $\beta=\beta_r(\alpha)$
and negative for all other $\beta>0$.
Hence $\mu_*(\alpha,\beta)$
is decreasing in $\beta$, as claimed.
Moreover, since $\lim_{\beta\to0^+}\mu_r^*(\alpha,\beta)=r$
and $\lim_{\beta\to\infty}\mu_r^*(\alpha,\beta)=-\infty$, 
$\beta_*(\alpha)$ is well-defined. 

We obtain the expression for $\mu_r^*(\alpha,\beta_r(\alpha))$ 
in the second claim
by \eqref{E_mu*}
and the equality $\alpha\beta_r(\alpha)^{r-1}/(r-1)!=1$. 
The conclusion
of the claim thus follows by the first claim, 
noting that  $\beta_r(\alpha)$ is decreasing in 
$\alpha$ and 
$\mu_r^*(\alpha_r,\beta_r(\alpha_r))=0$
since 
$\beta_r(\alpha_r)=(r/(r-1))^2$. 
\end{proof}
 
We are ready to prove the main result of this section. 

\begin{proof}[Proof of Proposition~\ref{P_SC}]
Let $\alpha<\alpha_r$ and $\delta>0$ be given. 
First, we show that with high probability, $\Gnp$ contains
no $m$-contagious set, for  $m=\beta\log{n}$
with $\beta\in[\beta_*(\alpha)+\delta,\beta_r(\alpha)]$. 

\begin{claim}\label{C_smallbeta}
For all $\beta\le\beta_r(\alpha)$, we have that
$\mu_r(\alpha,\beta,\gamma)\le \mu_r^*(\alpha,\beta)$. 
\end{claim}

This is proved in Appendix~\ref{A:C_smallbeta}.

By Lemmas~\ref{L_Eki},\ref{L_beta*} and Claim~\ref{C_smallbeta}, we find by summing over all $O(\log{n})$ relevant $k$
that the probability that such a set exists is bounded (up to a constant) by 
\[
n^{\mu_*(\alpha,\beta_*(\alpha)+\delta)}\log^{r(r-1)+1}n\ll1.
\]

It thus remains to show that with high probability, 
$\Gnp$ has no $m$-contagious set $I$, for some 
$m\geq \beta_r\log n$. To this end, note that if 
such a set $I$ exists, then 
there is some $t$ so that
  \[
    |V_t(I,\Gnp)| < \beta_r \log n \leq |V_{t+1}(I,\Gnp)|  
  \]
  Letting $k=|V_t(I,\Gnp)|$, we find that for some $k<\beta_r \log n$
  there is a $k$-contagious set $I$, with $m-k$ further vertices with $r$
  neighbours in $V_t(I,\Gnp)$.  

  The expected number of $k$-contagious sets with $i$ vertices 
  infected in the
  top level is $E_r(k,i)$.  Let $p_r(k,i)$ be the probability that for a
  given set of size $k$ with $i$ vertices identified as the top level,
  there are at least $\beta_r \log n - k$ vertices $r$- connected to the
  set with at least one neighbour in the top level. Hence the
  probability that $\Gnp$ has a $m$-contagious set $I$ for some
  $m\geq \beta_r\log n$ is at most
  \[
    \sum_{i<k<\beta_r(\alpha) \log n} E_r(k,i) p_r(k,i).
  \]
  The proposition now follows from the following claim,
  proved in Appendix~\ref{A:Epmax}. 

\begin{claim}\label{C:Epmax}
For all $k<\beta_r(\alpha)\log{n}$ and $i\le k-r$, we have that 
\[
E_{r}(k,i)p_{r}(k,i)\lesssim
n^{\mu_r^*(\alpha,\beta_r(\alpha))}\log^{r(r-1)}n
\]
where $\lesssim$ denotes inequality up to constant, 
independent of $i,k$.
\end{claim}

Indeed, by Claim~\ref{C:Epmax}, it follows, 
by summing over all $O(\log^2n)$ relevant $i,k$, 
that the probability that 
$\Gnp$ has an $m$-contagious set for some 
$m\ge\beta_r(\alpha)\log{n}$ 
is bounded (up to a constant) by 
\[
n^{\mu_r^*(\alpha,\beta_r)}\log^{r(r-1)+2}n
\ll1
\]
where the last inequality follows by Lemma~\ref{L_beta*}, 
since $\alpha<\alpha_r$ and hence 
$\mu_r^*(\alpha,\beta_r(\alpha) )<0$.
\end{proof}

\section{Upper bound for $p_c(n,r)$}\label{S_pcUB}

In this section, we prove Theorem~\ref{T_main}.  In light of Proposition~\ref{P_SC}, it
remains to prove that for $\alpha>\alpha_r$, with high probability
$\Gnp$ is susceptible.  Fundamentally this is done using the second
moment method.  
As discussed in the introduction, the main obstacle is showing that
contagious sets are sufficiently independent for the second moment method
to apply. 
To this end, we restrict to a
special type of contagious sets, which infect $k$ vertices with no
triangles.

As in the previous section, we fix $r\ge2$ throughout.

\subsection{Triangle-free susceptible graphs}\label{S_hat:m}

Recall that a graph is called \emph{triangle-free} if it contains no
subgraph which is isomorphic to $K_3$.

\begin{defn}
  Let $\hat m_r(k,i)$ denote the number of triangle-free graphs that
  contribute to $m_r(k,i)$ (see Section~\ref{S_m}).  Put
  $\hat m_r(k)=\sum_{i=1}^{k-r}\hat m_r(k,i)$.
\end{defn}

Following
Section~\ref{S_m}, we obtain a recursive lower bound for
$\hat m_r(k,i)$.  We note that $\hat m_r(k,k-r) = m_r(k,k-r) = 1$.  
For $i<k-r$ we claim that 
\begin{equation}\label{E_hat:mrec} 
\hat m_r(k,i) \ge {k-r\choose i}\sum_{j=1}^{k-r-i} \hat
a_r(k-i,j)^i \hat m_r(k-i,j)
\end{equation}
where 
\begin{equation}\label{E_hat:a} 
\hat a_r(x,y)=\max\{0,a_r(x,y)-2ryx^{r-2}\}.  
\end{equation}
Note that (in contrast to the recursion for
$m(k,i)$), this is only a lower bound.
To see \eqref{E_hat:mrec}, we argue that of the $a_r(k-i,j)$ ways to
connect a vertex in the top level to lower levels, at most $2rj(k-i)^{r-2}$
create a triangle.  This is so since the number of ways of choosing $r$
vertices from $k-i$, including at least one of the top $j$ and including at
least one edge is at most
\[
jr{k-i-2\choose r-2}+jr(k-i-r){k-i-3\choose r-3}
<2jr(k-i)^{r-2},
\] 
where the first term accounts for an edge including the previous top level
and the second term to $r$ vertices including an edge below the previous
top level.

Setting 
\[
\hat \sigma_{r}(k,i) 
= 
\frac{\hat m_{r}(k,i)}{(k-r)!}\left(\frac{(r-1)!}{k^{r-1}}\right)^k,
\] 
\eqref{E_hat:mrec} reduces to 
\begin{equation}\label{E_hat:sigrec}
\hat \sigma_{r}(k,i)
\ge\sum_{j=1}^{k-r-i} \hat A_r(k,i,j) \hat \sigma_{r}(k-i,j)
\end{equation}
where 
\begin{equation}\label{E_hat:A}
\hat A_r(k,i,j) 
=
\frac{j^i}{i!}\left(\frac{k-i}{k}\right)^{(r-1)k}
\left(\frac{(r-1)!}{(k-i)^{r-1}}\frac{\hat a_r(k-i,j)}{j}\right)^i.
\end{equation}

The following observation indicates that restricting to susceptible graphs
which are triangle-free does not have a significant effect on the
asymptotics.  

\begin{lem}\label{L_hat:A}
  Let $\hat A_r(k,i,j)$ be as in \eqref{E_hat:A} and let $A_r(i,j)$ be as
  defined in Lemma~\ref{L_A}.
  For any fixed $i,j\ge1$, we have that $\hat A_r(k,i,j)\to A_r(i,j)$,
  as $k\to\infty$.
\end{lem}

\begin{proof}
  Fix $i,j\ge1$.  From their definitions we have that
  \[
  \frac{\hat A_r(k,i,j)}{A_r(k,i,j)} = \left( \frac{\hat
    a_r(k,i,j)}{a_r(k,i,j)} \right)^i.
  \]
  Since $a_r(k,i,j)$ is of order $k^i$ and $\hat a_r(k,i,j)-a(k,i,j) =
  O(k^{i-1})$, we have $\hat a_r(k,i,j) / a_r(k,i,j) \to 1$.
  Since $i$ is fixed, it follows by Lemma~\ref{L_A} that
  \[
  \lim_{k\to\infty} \hat A_r(k,i,j)
  = \lim_{k\to\infty} A_r(k,i,j)
  = A_r(i,j).  \qedhere
  \]  
\end{proof}

In order to get asymptotic lower bounds on $\hat m_r(k,i)$ it is useful
to further restrict to graphs with bounded level sizes.

\begin{defn}
  For $\ell\ge r$, let $\hat m_{r,\ell}(k) \le \hat m_r(k)$ be the number
  of graphs that contribute to $\hat m_r(k)$ which have level sizes
  bounded by $\ell$ (i.e., $|I_i|\le\ell$ for all $i$).  Let
  $\hat m_{r,\ell}(k,i)$ be the number of such graphs with exactly
  $i\le \ell$ vertices in the top level.
  Hence $\hat m_{r,\ell}(k)=\sum_{i=1}^\ell \hat m_{r,\ell}(k,i)$.  
\end{defn}

Observe that for fixed $k$, $\hat m_{r,\ell}(k)$ is increasing in
$\ell$, and equals $m_r(k)$ for $\ell\geq k-r$.

Lemma~\ref{L_hat:A} will be used to prove asymptotic lower bounds for
$\hat m$.  When $i$ is small, the resulting bounds are not sufficiently
strong.  Thus we also make use of the following lower bound on
$\hat m_{r,\ell}(k,i)$ for values of $i$ which are small compared with
$k$.  This is also used as a base case for an inductive proof of lower
bounds using Lemma~\ref{L_hat:A}.

\begin{lem}\label{L_hat:b}
For all relevant $i,k$ and $\ell\ge r$ such that $k>r(r^2+1)+i+2$,
we have that 
\[
\hat m_{r,\ell}(k,i)
\ge{k-r\choose i}\hat b_r(k,i)^i \hat m_{r,\ell}(k-i)
\]
where
\[
\hat b_r(k,i)
=
{k-i-r-1\choose r-1}\left(1-\frac{r^3}{k-i-r-2}\right).
\]
In particular $\hat m_{r,\ell}(k,i)>0$ for such $k$.
\end{lem}

\begin{proof}
Let $i,k,\ell$ as in the lemma be given. 
We obtain the lemma by considering 
the subset ${\cal H}$ of 
graphs
contributing to $\hat m_{r,\ell}(k,i)$, constructed as follows. 
To obtain a graph $H\in {\cal H}$, 
select a subset $U\subset [k]-[r]$ of size $i$  
for the vertices in the top level of $H$, and a
minimally susceptible, triangle-free graph $H'$ on $[k]-U$ 
so that $[r]$ is a contagious set for $H'$ 
with all level sizes bounded by $\ell$
and 
$j$
vertices in the top level, for some $1\le j\le \min\{k-r-i,\ell\}$. 
Let $v$ denote the vertex in the top level of $H'$
of largest index. For each $u\in U$,
select a subset $V_u\subset [k]-U$ of size $r$ 
which contains $v$ and none of the neighbours of $v$ in 
$H'$ and so that no $v',v''\in V_u$ are connected
in $H'$. 
Finally, let $H$ be the minimally susceptible  graph 
on $[k]$ with
subgraph $H'$ such that each vertex $u\in U$
is connected to all vertices in $V_u$. 
By the choice of $H'$ and $V_u$, $H$ 
contributes to $\hat m_{r,\ell}(k,i)$. 
By the choice of $v$, 
for any choice of $U$, $H'$ and $V_u$, a unique graph $H$ 
is obtained. 
Hence $|{\cal H}|\le  \hat m_{r,\ell}(k,i)$. 

To conclude, we claim that,  
for each $u\in U$, the number of 
possibilities  for $V_u$ is bounded from below by
\[
{k-i-r-1\choose r-1}-r(k-i-r-1){k-i-r-3\choose r-3}
\ge
\hat b_r(k,i).
\]
To see this, note that of the $r(k-i)$ edges in $H'$, 
there are $r(r+1)$ that are either incident to $v$ or else connect a neighbour of
$v$ in $H'$ to another vertex below the top level of $H'$. 
Therefore 
\[
\hat m_{r,\ell}(k,i)
\ge
{k-r\choose i}\hat b_r(k,i)
\sum_{j}\hat m_{r,\ell}(k-i,j)
=
{k-r\choose i}\hat b_r(k,i)\hat m_{r,\ell}(k-i)
\]
(where the sum is over $1\le j\le \min\{k-r-i,\ell\}$)
as claimed. 

By the choice of $i,k$, $\hat b_r(k,i)>0$.  
Hence $\hat m_{r,\ell}(k,i)>0$,
since
 $\hat m_{r,\ell}(k)>0$ for all relevant $k,\ell$,
 as is easily seen (e.g., by considering minimally susceptible, 
 triangle-free
 graphs of size $k=nr+m$, for some $n\ge1$ and $m\le r$, 
 which have $m$ vertices in the top level and $r$ vertices in 
 all levels below, and 
 all vertices in level $i\ge1$ are connected to all $r$ 
 vertices in level $i-1$).
\end{proof}

\begin{lem}\label{L_hat:mLB}
  As $k\to\infty$, we have that 
  \[
    m_r(k) \geq \hat m_r(k) \ge 
    e^{-o(k)} e^{-(r-2)k}(k-r)!\left(\frac{k^{r-1}}{(r-1)!}\right)^k.
  \]
\end{lem}

Comparing this with Lemma~\ref{L_mUB}, we see that the number of
triangle-free susceptible graphs of size $k$ is not much smaller than the
number of susceptible graphs (up to an error of $e^{o(k)}$).

\begin{proof}
  The idea is to use spectral analysis of the linear recursion
  \eqref{E_hat:sigrec-ell}.  However, some work is needed to write the
  recursion in a usable form.
  
  Put
  \[
    \hat \sigma_{r,\ell}(k,i) 
    = 
    \frac{\hat m_{r,\ell}(k,i)}{(k-r)!}\left(\frac{(r-1)!}{k^{r-1}}\right)^k
  \]
  Restricting \eqref{E_hat:sigrec} to $j\le \ell$, it follows that  
  \begin{equation}\label{E_hat:sigrec-ell}
    \hat \sigma_{r,\ell}(k,i)
    \ge\sum_{j=1}^{\ell} \hat A_r(k,i,j) \hat \sigma_{r,\ell}(k-i,j)
    \quad \text{for $i\le \ell$}.
  \end{equation}

  In order to express \eqref{E_hat:sigrec-ell} in matrix form, 
  we introduce the following notations. 
  For an $\ell\times \ell$ matrix $M$, let 
  $M_j$, be the $\ell\times\ell$ matrix whose $j$th row
  is that of $M$ and all other entries are 0.  Let 
  \[
    \psi(M)=  
    \begin{bmatrix}
      M_1 	& M_2 & \cdots & M_{\ell-1} & M_\ell \\
      I_\ell 	\\
      & I_\ell \\
      & & \ddots\\
      & & & I_\ell
    \end{bmatrix}
\]
where $I_\ell$ is the $\ell\times\ell$ identity matrix
and all empty blocks are filled with 0's. 
For all relevant $k$, put  
\[
\hat \Sigma_k=
\hat \Sigma_k(r,\ell)=
\begin{bmatrix}
\hat \sigma_k\\
\hat \sigma_{k-1}\\
\vdots\\
\hat \sigma_{k-\ell+1}
\end{bmatrix}
\]
where $\hat \sigma_k=\hat \sigma_k(r,\ell)$ is the $1\times \ell$ vector 
with entries $(\hat \sigma_k)_j = \hat \sigma_{r,\ell}(k,j)$. 

Using this notation, \eqref{E_hat:sigrec-ell} can be written as
\begin{equation}\label{E_hat:Sig}
\hat \Sigma_k\ge\psi(\hat A_{k}) \hat \Sigma_{k-1},
\end{equation}
where $\hat A_k=\hat A_k(r,\ell)$ 
is the $\ell\times\ell$ matrix 
with entries $(\hat A_k)_{i,j} = \hat A_r(k,i,j)$. 

By Lemma~\ref{L_hat:b}, we have that all coordinates of 
$\hat\Sigma_k$ are positive for all $k$ large enough.
Let $A=A(r,\ell)$ denote the $\ell\times\ell$ matrix 
with entries $A_{i,j}= A_r(i,j)$. 
For $\eps>0$, let $A_\eps=A_\eps(r,\ell)$, 
be the $\ell\times\ell$ matrix 
with entries $(A_\eps)_{i,j}= A_{i,j}-\eps$. By 
Lemma~\ref{L_hat:A}, for $k$ large enough each entry of $\hat A_k$ is
greater than the same entry of $A_\eps$.
Since $A>0$, for some $\eps_{r,\ell}>0$, we have that 
$A_\eps>0$ for all $\eps\in(0,\eps_{r,\ell})$. 
Hence, by Lemma~\ref{L_hat:A} and \eqref{E_hat:Sig}, 
for any such $\eps>0$, there is a $k_\eps$ so that 
\[
\hat\Sigma_{k_\eps+k}\ge\psi(A_\eps)^k \hat\Sigma_{k_\eps}>0
\quad \text{for $k\ge0$},
\]
with entries of $\Sigma_{k_\eps}$ positive.
Therefore, up to a factor of $e^{-o(k)}$, 
the growth rate of 
$\hat \sigma_{r,\ell}(k)
=\sum_{i} \hat \sigma_{r,\ell}(k,i)$
is given by the Perron 
eigenvalue $\lambda=\lambda(r,\ell)$ of $\psi(A)$.

Let $D_\lambda = {\rm diag}(\lambda^{-i}:1\le i\le \ell)$.  We claim
that the Perron eigenvalue of $\psi(A)$ is characterized by the property
that the Perron eigenvalue of $D_\lambda A$ is $1$.  To see this, one
simply verifies that if $D_\lambda A v = v$, then 
\[
v_\lambda
=
\begin{bmatrix}
\lambda^{\ell-1}v\\
\lambda^{\ell-2}v\\
\vdots\\
v
\end{bmatrix}
\]
satisfies $\psi(A) v_\lambda = \lambda v_\lambda$.  If $v$ has
non-negative entries, then $1$ is the Perron eigenvalue of $D_\lambda A$
and $\lambda$ the Perron eigenvalue of $\psi(A)$.

Fix $\delta>0$. If $\lambda < e^{-(r-2)}(1-\delta/e)$, we claim that
for $\ell$ large enough, every row sum of $D_\lambda A$ is
greater than $1$. Indeed, the sum of row $i\le \ell$ is
(using the bound $i!<(i/e)^i$)
\[
(e^{r-1} \lambda )^{-i} \sum_{j=1}^\ell \frac{j^i}{i!} 
\geq (e^{r-1} \lambda )^{-i} \frac{\ell^i}{i!}
\geq (e-\delta)^{-i} \frac{\ell^i}{i!}
\geq \left(\frac{e}{e-\delta}\frac{\ell}{i}\right)^i
\geq \left(\frac{e}{e-\delta}\right)^i>1.
\]
Since the spectral radius of a matrix is bounded below by its minimum
row sum, it follows that for such $\lambda$, the spectral radius of
$D_\lambda A$ is greater than $1$.  Since the spectral radius of
$D_\lambda A$ is decreasing in $\lambda$, the Perron eigenvalue of
$\psi(A)$ is at least $e^{-(r-2)}(1- \delta/e)$ for $\ell$ large enough, and
hence $\liminf_{\ell\to\infty} \lambda(r,\ell) \geq e^{-(r-2)}$.
Taking $\ell\to\infty$, we find that    
\[
\hat m_r(k) \ge 
e^{-o(k)} e^{-(r-2)k}(k-r)!\left(\frac{k^{r-1}}{(r-1)!}\right)^k
\]
as claimed. 
\end{proof}

We require a lower bound for the number of 
minimally susceptible  graphs of size $k$ 
with $i=\Omega(k)$ vertices in the top level
in order to estimate the growth of super-critical 
$r$-percolations on $\Gnp$. 

\begin{lem}\label{L_hat:mkiLB}
Let $\eps\in(0,1/(r+1))$. 
For all sufficiently large $k$ and $i\le (\eps/r)^2k$,
we have that  
\[
\hat m_r(k,i)\ge 
e^{-i\eps-(r-2)k-o(k)}(k-r)!\left(\frac{(k-i)k^{r-2}}{(r-1)!}\right)^k
\] 
where $o(k)$ depends on $k,\eps$, 
but not on $i$. 
\end{lem}

Although the proof is somewhat involved, 
the general scheme is straightforward. 
We use Lemmas~\ref{L_hat:b},\ref{L_hat:mLB} 
to obtain a sufficient bound for $i,k$ 
in a range for which $i/k\ll1$. 
Then, for all other
relevant $i,k$ we proceed by induction, using \eqref{E_hat:mrec}. 
The inductive step (Claim~\ref{Cl_hat:rhoind} below) of the proof 
appears in  Appendix~\ref{A_Cl_hat:rhoind}.

\begin{proof}
Fix some $k_r$ so that 
\[
k_r>
\max\left\{e^{r/\eps},\frac{r(r^2+1)+2}{1-(\eps/r)^{2}}\right\}.
\]
Note that, for all $k>k_r$ and $i\le(\eps/r)^2k$, 
we have that 
$k/\log^2k< (\eps/r)^2k$
and that Lemma~\ref{L_hat:b} 
applies to $\hat m_r(k,i)$ 
(setting $\ell=k-r$, so that 
$\hat m_{r,\ell}(k,i)=\hat m_r(k,i)$). 

For all relevant $i,k$, let  
\begin{equation}\label{E_hat:rho}
\hat \rho_r(k,i)=\frac{\hat m_r(k,i)}{(k-r)!}
\left(\frac{(r-1)!}{(k-i)k^{r-2}}\right)^k.
\end{equation}

By Lemma~\ref{L_hat:mLB} there is some $f_r(k)\ll k$ such that 
\[
\hat m_r(k)\ge e^{-(r-2)k-f_r(k)}(k-r)!\left(\frac{k^{r-1}}{(r-1)!}\right)^k.
\]  
Without loss of generality, we assume 
$f_r$ is non-decreasing. 

By Lemma~\ref{L_hat:b}, 
we find that for all $k>k_r$ and relevant $i$,  
$\hat \rho_r(k,i)$ is bounded from below by 
\[
\frac{e^{-(r-2)(k-i)-f_r(k-i)}}{i!}
\hat b_r(k,i)^i
\left(\frac{(k-i)^{r-1}}{(r-1)!}\right)^{k-i}
\left(\frac{(r-1)!}{(k-i)k^{r-2}}\right)^k.
\]
By the bound ${n\choose k}\ge (n-k)^k/k!$,
\[
\hat b_r(k,i)
\ge
\frac{(k-i-2r)^{r-1}}{(r-1)!}
\left(1-\frac{r^3}{k-i-r-2}\right).
\]
Therefore
the lower bound for $\hat \rho_r(k,i)$ above 
is bounded from below by (using the 
inequality $i!<i^i$)
\[
C_r(k,i)
g_r(k,i)e^{-(r-2)k-f_r(k-i)-i\log{i}}
\]
where 
\[
C_r(k,i)
=
\left(1-\frac{2r}{k-i}\right)^{(r-1)i}
\left(1-\frac{r^3}{k-i-r-2}\right)^i
\]
and
\[
g_r(k,i)
=
e^{(r-2)i}\left(\frac{k-i}{k}\right)^{(r-2)k}.
\]

If $r=2$, then $g_r\equiv1$. We note that, for $r>2$,  
\[
\frac{\partial}{\partial i}g_r(k,i)
=-\frac{(r-2)i}{k-i}g_r(k,i)<0
\]
and so, for any such $r$, for any relevant $k$, 
$g_r(k,i)$ is decreasing in $i$.
By the inequality $(1-x)^y>1-xy$, 
for any $k>k_r$ and $i\le (\eps/r)^2k$, 
\begin{align*}
C_r(k,i)
&>1-\frac{r^2i}{k-i}-\frac{r^3i}{k-i-r-2}\\
&>1
-\frac{2\eps^2}{1-(\eps/r)^2}
-\frac{r\eps^2}{1-(\eps/r)^2-(r+2)/k}\\
&>1
-\frac{2/(r+1)^2}{1-1/r^4}
-\frac{1/r}{1-1/r^4-(r+2)/k_r}\\
&>0
\end{align*}
since $k_r>e^{r/\eps}>e^{r(r+1)}$, $r\ge2$, and $\eps<1/(r+1)$
(and noting that the second last line is increasing in $r$). 
Altogether, for some $\xi'(r)>0$, 
we have that 
\begin{equation}\label{E_hat:rhoLB-smalli}
\hat \rho_r(k,i)\ge 
\xi'(r) e^{-(r-2)k-h_r(k)}\quad \text{for $k>k_r$ and $i\le k/\log^2{n}$}
\end{equation}
where 
\begin{equation}\label{E_hrk}
h_r(k)=f_r(k)
-\log g_r\left(k,\frac{k}{\log^2{k}}\right)
+\frac{k}{\log^2{k}}\log\left(\frac{k}{\log^2{k}}\right).
\end{equation}

We note that $h(k)\ll k$ as
$k\to\infty$.

\begin{claim}\label{Cl_hat:rhoind}
For some $\xi=\xi(r,\eps)>0$, for all $k>k_r$ 
and $i\le (\eps/r)^2k$, we have that 
$\hat \rho_r(k,i) \ge \xi e^{-i\eps-(r-2)k-h_r(k)}$. 
\end{claim}

Claim~\ref{Cl_hat:rhoind} is proved in Appendix~\ref{A_Cl_hat:rhoind}.

Since $h_r(k)\ll k$ and $\xi$ depends only on $r,\eps$, 
the lemma follows by 
Claim~\ref{Cl_hat:rhoind} and \eqref{E_hat:rho}.
\end{proof}

\subsection{$\hat r$-bootstrap percolation on $\Gnp$}
\label{S_hat:Pki}

We define \emph{$\hat r$-percolation,}
a restriction of $r$-percolation,
which informally halts upon requiring a triangle.
Formally, recall the definitions of $I_t(I,G)$ and 
$V_t(I,G)$ given in Section~\ref{S_m}. 
Let $\hat I_t=I_t$ if $G$ contains a triangle-free subgraph
$H$ such that $V_t(I,H)=V_t(I,G)$, and otherwise put 
$\hat I_t=\emptyset$. 
Put $\hat V_t=\bigcup_{s\le t}\hat I_s$. 

\begin{defn}
Let $\hat P_r(k,i)=\hat P_r(p,k,i)$, for some $p=p(n)$,  
denote the probability 
that for a given $I\subset [n]$, with $|I|=r$, 
we have 
that $|\hat V_t(I,\Gnp)|=k$
and $|\hat I_t(I,\Gnp)|=i$, for some  $t$. 
Let $\hat E_r(k,i)$ denote the expected 
number of such subsets $I$. 
We put $\hat P_r(k)=\sum_{i=1}^{k-2} \hat P_r(k,i)$ and  
$\hat E_r(k)=\sum_{i=1}^{k-r}\hat E_r(k,i)$.
\end{defn}

Using Lemma~\ref{L_hat:mkiLB}, we obtain 
lower bounds on the growth probabilities 
of $\hat r$-percolations on $\Gnp$. 

\begin{lem}\label{L_hat:Pki}
Let $\alpha>0$.
Put $p=\vartheta(\alpha,n)$
and $\eps=np^r=\alpha/\log^{r-1}{n}$. 
For $i\le k-r$ and $k\le n^{1/(r(r+1))}$, we have that  
\[
\hat P_r(k,i)
\ge 
(1-o(1))\frac{e^{-\eps{k-i\choose r}}\eps^{k-r}}{(k-r)!} \hat m_r(k,i)
\]
where $o(1)$ depends on $n$, but not on $i,k$.  
\end{lem}

\begin{proof}
Let $I\subset[n]$, with $|I|=r$, be given. Put 
\[
\hat\ell_r(k,i)
=\frac{e^{-\eps{k-i\choose r}}\eps^{k-r}}{(k-r)!} \hat m_r(k,i).
\]
If for some $V\subset[n]$ 
with $|V|=k$ and $I\subset V$ we have that 
the subgraph 
$G_V\subset\Gnp$ induced by $V$ is minimally susceptible
and triangle-free,  
$I$ is a contagious set for $G_V$ with $i$ vertices in the top level,
and all vertices in $v\in V^c$ are connected to at most
$r-1$ vertices below the top level of $G_V$, 
then it follows that  $|\hat V_t(I,\Gnp)|=k$
and $|\hat I_t(I,\Gnp)|=i$ for some $t$.
Hence  
\[
\hat P_r(k,i)>
{n-r\choose k-r}\hat m_r(k,i)p^{r(k-r)}(1-p)^{k^2}
\left(1-{k-i\choose r}p^r\right)^n.
\]  

By the inequalities ${n\choose k}\ge(n-k)^k/k!$ and 
$(1-x/n)^n\ge e^{-x}(1-x^2/n)$,
it follows that  
\[
\frac{\hat P_r(k,i)}{\hat \ell_r(k,i)}
>
\left(1-\frac{k}{n}\right)^k(1-p)^{k^2}
\left(1-{k-i\choose r}^2\frac{\eps^2}{n}\right)
\]
For all large $n$, the right hand side is bounded from below by 
\[
\left(1-\frac{k}{n}\right)^k
\left(1-\frac{1}{n^{1/r}}\right)^{k^2}
\left(1-\frac{k^{2r}}{n}\right)\sim1
\]
since $k\le n^{1/(r(r+1))}\ll n^{1/(2r)}$, as $r\ge2$. 
It follows that $\hat P_r(k,i)\ge(1-o(1))\hat \ell_r(k,i)$, 
where $o(1)$ depends on $n$, 
but not on $i,k$, as required. 
\end{proof}

\subsection{Super-critical bounds}
\label{S_hat:sup-rperc}

In this section  
we show that, for $\alpha>\alpha_r$, 
the expected number of 
super-critical $\hat r$-percolations  on $\Gnp$
which grow larger than a critical size 
$\beta_*(\alpha)\log{n}>\beta_r(\alpha)\log{n}$ 
 is large.
The importance of $\beta_*(\alpha)$ is 
established in Section~\ref{S_term} below.
Subsequent sections establish
the existence of sets $I$ of size $r$
so that $\hat r$-percolation initialized at $I$
grows larger than $\beta_*(\alpha)\log{n}$. 

\begin{lem}\label{L_hat:EkiLB}
Let $\alpha,\beta_0>0$ and $\eps\in(0,1/(r+1))$.
Put $p=\vartheta(\alpha,n)$. For all sufficiently large 
$k=\beta\log{n}$ 
and $i=\gamma k$, with $\beta\le\beta_0$ 
and $\gamma\le(\eps/r)^2$,  
we have that  
\[
\hat E_r(k,i)\ge n^{\mu_\eps-o(1)}
\] 
where
\[
\mu_\eps=
\mu_{r,\eps}(\alpha,\beta,\gamma)
=
r+
\beta\log\left(\frac{\alpha\beta^{r-1}(1-\gamma)}{(r-1)!}\right)
-\frac{\alpha\beta^r}{r!}(1-\gamma)^r
-\beta(r-2+\eps\gamma) 
\] 
and $o(1)$ depends on $\alpha,\eps,\beta_0$, 
but not on $\beta,\gamma$. 
\end{lem}

\begin{proof}
Put $\delta=np^2$. 
By 
Lemmas~\ref{L_hat:mkiLB},\ref{L_hat:Pki}, 
for large $k=\beta\log{n}$ 
and $i=\gamma k$, with $\beta\le\beta_0$ 
and $\gamma\le(\eps/r)^2$,
\[
\hat E_r(k,i)
\ge
\xi(n){n\choose r}
\left(\frac{\delta (k-i)k^{r-2}}{(r-1)!}\right)^k \delta^{-r}
e^{-i\eps-(r-2)k-\delta{k-i\choose r}-o(k)} 
=n^{\mu_\eps-o(1)}
\]
where 
$\xi(n)\sim1$ depends only on $n$, and 
$o(k)$ depends only on $r,\eps,\beta_0$.
\end{proof}

We note that, for any $\alpha,\eps>0$, 
\begin{equation}\label{E_mueps-gam0}
\mu_{r,\eps}(\alpha,\beta,0)=\mu_r^*(\alpha,\beta). 
\end{equation}

We now state the main result of this section. 

\begin{lem}\label{L_mueps}
Let $\eps<1/(r+1)$. Put 
$\alpha_{r,\eps}=(1+\eps)\alpha_r$ 
and $p=\vartheta(\alpha_{r,\eps},n)$. 
For some $\delta(r,\eps)>0$ and $\zeta(r,\eps)>0$ 
we have that 
if $k_n/\log{n}\in[\beta_*(\alpha_{r,\eps}), 
\beta_*(\alpha_{r,\eps})+\delta]$ 
for all large $n$, then 
$\hat E_{r}(k_n)\gg n^{\zeta}$ as $n\to\infty$. 
\end{lem}

The proof appears in Appendix~\ref{A_L_mueps}.
The argument is technical but straightforward: the basic idea
is to show that, for some $\zeta>0$ and all large $n$,  
for all relevant $k$ there is some $i$
so that $\hat E_r(k,i)>n^{\zeta}$. 
For $k>\beta_*\log{n}$, values of $i$ 
with this property are on the order of $k$. 
We shall thus require Lemma~\ref{L_hat:mkiLB}.

\subsection{$\hat r$-percolations are almost independent}
\label{S_hat:rperc-ai}

For a set $I\subset[n]$, with $|I|=r$, let $\hat {\cal E}_k(I)$ denote
the event that $\hat r$-percolation on $\Gnp$ initialized by $I$ grows
to size $k$, i.e., we have that $|\hat V_t(I)|=k$ for some $t$.  Hence
$\hat P_r(k)=\P(\hat {\cal E}_k(I))$. 
In this section we show that for sets $I\neq I'$ of size $r$ and
suitable values of $k,p$, the events $\hat {\cal E}_k(I)$ and
$\hat {\cal E}_k(I')$ are approximately independent.  Specifically, we
establish the following

\begin{lem}\label{L_hat:cr2mm}
  Let $\alpha,\beta>0$ and put $p=\vartheta_r(\alpha,n)$.  Fix sets
  $I\neq I'$ such that $|I|=|I'|=r$ and $|I\cap I'|=m$.  For
  $\beta\log{n}\le k\le n^{1/(r(r+1))}$, we have that
  \[
    \P(\hat {\cal E}_k(I')|\hat {\cal E}_k(I)) \le 
    \left(\tfrac{k}{n}\right)^{r-m} 
    +O\big(k^{2r}(kp)^{r(r-m)}\big) + 
    \begin{cases}
      (1+o(1))\hat P_r(k) &\text{if $m=0$,} \\
      o(\left(\tfrac{n}{k}\right)^{m})\hat P_r(k) 
      &\text{if $1\leq m < r$,} \\
    \end{cases}
  \]
  where $o(1)$ depends only on $n$.
\end{lem}

\newcommand{\cE}{{\cal E}}

For sets $I\subset V$ of sizes $r$ and $k$, let $\hat\cE(I,V)$ be the
event that for some $t$ we have $\hat V_t(I)=V$.  By symmetry these
events all have the same probability. Since for a fixed $I$ and
different sets $V$ these events are disjoint, we have
$\hat P_r(k) = \binom{n-r}{k-r} \P(\hat\cE(I,V))$.

\begin{lem}\label{L_hat:EV}
Fix sets $I\subset V$ with $|I|=r$ and $|V|=k$. 
\begin{enumerate}[nolistsep,label={(\roman*)}]

\item For any set of edges $E\subset[n]^2-V^2$, the conditional
  probability that $E\subset E(\Gnp)$, given $\hat \cE(I,V)$, is at most
  $p^{|E|}$.

\item For any $u\notin V$ and set of vertices $W\subset[n]$ such that
  $|W|=r$ and $|V\cap W|<r$, the conditional probability that
  $(u,w)\in\Gnp$ for all $w\in W$, given $\hat \cE(I,V)$, is at least
  $p^r(1-p)^k$.
\end{enumerate}
\end{lem}

\begin{proof}
  Let $\G_V$ denote the subgraph of $\Gnp$ induced by $V$.  The
  event $\hat {\cal E}(I,V)$ occurs if and only if for some $t$ and
  triangle-free subgraph $H\subset \G_V$, we have that
  $V_t(I,H)=V_t(I,\G_V)=V$ and all vertices in $V^c$ are connected
  to at most $r-1$ vertices below the top level of $H$ (i.e.,
  $V-I_t(I,H)$).  This event is increasing in the set of edges of
  $\G_V$, and decreasing in edges outside $V$.  By the FKG inequality, 
  \[
    \P(E\subset E(\Gnp) | \hat {\cal E}(I,V)) \leq \P(E\subset E(\Gnp))
    = p^{|E|}.
  \]

  For claim (ii), let $G$ be a possible value for $\G_V$ on
  $\hat\cE(I,V)$, with a subgraph $H$ as above and $i\le k-r$ vertices
  infected in the top level (i.e., $I_t(I,H)=i$).  The conditional
  probability that $u$ is connected to all vertices in $W$, given
  $\hat \cE(I,V)$ and $\G_V=G$, is equal to
\[
\frac{p^r\sum_{\ell=0}^{r-1-\ell_0}{k-i-\ell_0\choose\ell}p^\ell(1-p)^{k-i-\ell_0-\ell}}
{\sum_{\ell=0}^{r-1}{k-i\choose\ell}p^{\ell}(1-p)^{k-i-\ell}}
\]
where $\ell_0<r$ is the number of vertices in $W$ 
below the top level of $H$. 
Bounding the numerator by the $\ell=0$ term 
and the denominator by 1, the
above expression is at least 
$p^r(1-p)^{k-i-\ell_0}\ge p^r(1-p)^{k}$. 
Hence, summing over the possibilities 
for $G$ we obtain the second claim. 
\end{proof}

The following result, 
a special case of Tur\'an's Theorem \cite{T41}, 
plays an important role in establishing
the approximate independence of $\hat r$-percolations. 

\begin{lem}[{Mantel's Theorem~\cite{M07}}]\label{L_Mantel}
If a graph $G$ is triangle-free, then we have that 
$e(G)\le \lfloor v(G)^2/4 \rfloor$. 
\end{lem}

In other words, a triangle-free graph has edge-density at most $1/2$. 
The number $2r-1$ is key, since   
$\lfloor (2r-1)^2/4 \rfloor=r(r-1)$, and thus 
\begin{equation}\label{E_nu}
r(2r-1)-
\lfloor (2r-1)^2/4 \rfloor 
=r^2. 
\end{equation}

\begin{lemma}\label{L_corr}
Let $\alpha>0$ and $k\le n^{1/(r(r+1))}$.  
Put $p=\vartheta_r(\alpha,n)$. 
  Fix sets $I\subset V$ and $I'$ such that 
  $|I|=|I'|=r$, $|V|=k$
  and $\ell=|V\cap I'|<r$.  
  Let $\hat\cE_{k,q}(I')$ denote the event that for some 
  $t$ we have that $\hat V_{t}(I')=V'$ for some $V'$
  such that $|V'|=k$ and $|V\cap V'|=q$.
  Then
  \[
    \P\big(\hat\cE_{k,q}(I')|\hat\cE(I,V) \big) \le  
    \begin{cases}
      (1+o(1)) \hat P_r(k) & q=0, \\
      o((n/k)^\ell) \hat P_r(k) & 1\leq q < 2r-1, \\
      k^{2r-1}(kp)^{r(r-\ell)} & q\geq 2r-1,
    \end{cases}
  \]
  where $o(1)$ depends only on $n$. 
\end{lemma}

\begin{proof}
{\bf Case i} ($q<2r-1$).
We claim that  
\begin{equation}\label{E_Ek-smallq}
\P\big(\hat\cE_{k,q}(I')|\hat\cE(I,V) \big)\le 
\left(\left(\frac{n}{k}\right)^{\ell}\left(\frac{k^2}{np^{q/4}}\right)^{q}\right)
\sum_{i=1}^{k-r}
\hat Q_r(k,i)
\end{equation}
where $\hat Q_r(k,i)$ is equal to 
\[
{n\choose k-r}
\hat m_r(k,i)
p^{r(k-r)}
\left(
1
-
\left({k-i\choose r}-{q\choose r}\right)p^r(1-p)^{2k}
\right)^{n-2k}.
\]
To see this, note that 
 if $\hat\cE_{k,q}(I)$ occurs then 
for some $V'$ such that $|V'|=k$, $I'\subset V'$,
and $|V\cap V'|=q$, we have that 
$I'$ is a contagious set for a triangle-free subgraph 
$H'\subset\Gnp$ on $V'$ with 
$i$ vertices in the top level, for some $i\le k-r$, 
and all vertices in $(V\cup V')^c$
are connected to at most $r-1$ vertices 
below the top level of $H'$. 
There are at most 
\[
{k\choose q-\ell}
{n-(q-\ell)\choose k-r-(q-\ell)}
\le
\left(\frac{n}{k}\right)^{\ell}
\left(\frac{k^2}{n}\right)^q
{n\choose k-r}
\]
such subsets $V'$. 
By Lemma~\ref{L_hat:EV},\ref{L_Mantel}, 
for any such $V'$ and $i$ as above, 
the conditional probability that such a subgraph $H'$ exists,
given $\hat \cE(I,V)$, 
is bounded by $\hat m(k,i)p^{r(k-r)-q^2/4}$,
since at most $q^{2}/4$ edges of $H'$ join
vertices in $V\cap V'$. 
By Lemma~\ref{L_hat:EV},
for any $u\in (V\cup V')^c$ and  
set $V''$ 
of $r$ vertices below the top level of $H'$ 
with at most $r-1$ vertices in $V\cap V'$, 
the conditional probability that $u$ is connected to 
all vertices in $V''$ 
is at least $p^r(1-p)^k$. 
Hence any such $u$ is 
connected to all vertices in such a $V''$
with conditional probability at least 
$\left({k-i\choose r}-{q\choose r}\right)p^r(1-p)^{2k}$.
The claim follows. 

To conclude, let $\hat \ell_r(k,i)$ 
be as in the proof of Lemma~\ref{L_hat:Pki}, 
which recall shows that 
$\hat P_r(k,i)\ge(1-o(1))\hat\ell_r(k,i)$ as $k\to\infty$, 
where $o(1)$ depends only on $n$. 
We have, by the inequalities 
${n\choose k}\le n^k/k!$ and $1-x<e^{-x}$, 
that 
\[
\log \frac{\hat Q_r(k,i)}{\hat \ell_r(k,i)}
<
\eps{k-i\choose r}
\left(1-(1-p)^{2k}\left(1-\frac{2k}{n}\right)\right)+\eps q^r.
\]
By the inequality $(1-x)^y\ge1-xy$, 
and since $k\le n^{1/(r(r+1))}$, 
it follows that 
the right hand side is at most
$\eps n^{1/r}(p+1/n) +\eps q^r\sim0$, and so 
\[
\hat Q_r(k,i)
\le (1+o(1))\hat \ell_r(k,i)
\le (1+o(1))\hat P_r(k,i)
\]
where $o(1)$ depends only on $n$.
Hence
\[
\sum_{i=1}^{k-r}
\hat Q_r(k,i)
\le (1+o(1))\sum_{i=1}^{k-r}\hat P_r(k,i)
= (1+o(1))\hat P_r(k).
\]
Finally, case (i) follows by \eqref{E_Ek-smallq}
and noting that 
\[
\frac{np^{q/4}}{k^2}
>\frac{np^{r/2}}{k^2}
\ge n^{1/2-2/(r(r+1))}\left(\frac{\alpha }{\log^{r-1}n}\right)^{1/2}
\gg 1
\]
since  
$q<2r$, $k\le n^{1/(r(r+1))}$ and $r\ge2$.

{\bf Case ii} ($q\ge2r-1$). 
Put $q_*=2r-1-\ell$. 
If $\hat {\cal E}_{k,q}(I')$ occurs,  
then 
for some $\{v_j\}_{j=1}^{q_*}\subset V-I'$ and  
non-decreasing sequence $\{t_j\}_{j=1}^{q_*}$, 
we have that 
$v_j\in \hat I_{t_j}(I')$  
and 
$\hat V_j=\hat V_{t_j-1}(I')$  
satisfy $|\hat V_{q_*}|<k$
and  $\hat V_j\cap (V-I')\subset  \bigcup_{i<j}\{v_i\}$.
Informally, $t_j$ is the $j$th time 
that $\hat r$-percolation initialized by 
$I'$ infects a vertex in $V-I'$.
It follows that 
$\Gnp$ contains a triangle-free subgraph on 
$\{v_j\}_{j=1}^{q_*} \cup \hat V_{q_*}$. 
Since $v_j\in \hat I_{t_j}(I')$, note that $v_j$ is  
$r$-connected to $\hat V_j$. 
Hence, by Lemma~\ref{L_Mantel}
and \eqref{E_nu},  
there are at least 
\[
rq_*
-\lfloor (2r-1)^2/4 \rfloor 
=
r(r-\ell)
\] 
edges between $\{v_j\}_{j=1}^{q_*}$
and $\hat V_{q_*}-V$. Thus, by Lemma~\ref{L_hat:EV}, 
 the conditional probability of 
$\hat {\cal E}_{k,q}(I')$,
given $\hat {\cal E}(I,V)$, is bounded by 
$k^{q_*}(kp)^{r(r-\ell)} \le k^{2r-1}(kp)^{r(r-\ell)}$,
as claimed.  
\end{proof}

Using Lemma~\ref{L_corr} 
we establish the main result of this section.  

\begin{proof}[Proof of Lemma~\ref{L_hat:cr2mm}]
  Fix a sequence of sets $\{V_\ell\}_{\ell=m}^r$ such that
  $I\subset V_\ell$ and $\ell=|V_\ell\cap I'|$.  By symmetry, we have
  that
  \begin{align*}
    \P(\hat\cE_k(I')|\hat\cE_k(I))
    &=
    {n-r\choose k-r}^{-1}
    \sum_{\ell=m}^r{n-r-(\ell-m)\choose k-r-(\ell-m)}\P(\hat\cE_k(I')|\hat\cE(I,V_\ell))\\
    &\le \sum_{\ell=m}^r(k/n)^{\ell-m}\P(\hat\cE_k(I')|\hat\cE(I,V_\ell)).
  \end{align*}
  If $\ell=m$, then by Lemma~\ref{L_corr}, summing over $q\in[\ell,k]$, we get
\[
\P(\hat\cE_k(I')|\hat\cE(I,V_m)) 
\le  
\begin{cases}
(1+o(1))\hat P_r(k) + k^{2r}(kp)^{r^2}  & m=0, \\
o((n/k)^{m})\hat P_r(k) + k^{2r}(kp)^{r(r-m)} & 1\leq m < r. \\
\end{cases}
\] 

Likewise,
for any $m< \ell<r$, 
\begin{align*}
(k/n)^{\ell-m}
\P(\hat\cE_k(I')|\hat\cE(I,V_\ell))
&\le 
(k/n)^{\ell-m}
\left(
o((n/k)^{\ell})\hat P_r(k)+k^{2r}(kp)^{r(r-\ell)}
\right)\\
&= 
o((n/k)^{m})\hat P_r(k)
+k^{2r}(kp)^{r(r-m)}(np^rk^{r-1})^{m-\ell}\\
&\le 
o((n/k)^{m})\hat P_r(k)
+k^{2r}(kp)^{r(r-m)}(\alpha\beta^{r-1})^{m-\ell}\\
&=
o((n/k)^{m})\hat P_r(k)
+O(k^{2r}(kp)^{r(r-m)}).
\end{align*}
Finally, for $\ell=r$ we bound $\P(\hat\cE_k(I')|\hat\cE(I,V_r))\leq 1$.
Summing over $\ell\in[m, r]$ we obtain the result.
\end{proof}

\subsection{Terminal $r$-percolations}\label{S_term}

In this section, we establish the importance of 
$\beta_*(\alpha)$
to the growth of super-critical $r$-percolations. 
Essentially, we find that an $r$-percolation on $\Gnp$, 
having grown larger than $\beta_*(\alpha)\log{n}$,  
with high probability continues to grow. 

\begin{defn}
We say that $I\subset[n]$ is a \emph{terminal $(k,i)$-contagious
set} for $\Gnp$ if 
$|V_{\tau}(I,\Gnp,r)|=k$ and 
$|L_{\tau}(I,\Gnp,r)|=i$.
\end{defn}
 
\begin{lem} \label{L_term}
Let $\alpha>\alpha_r$ and 
$\beta_r^*(\alpha)<\beta_1<\beta_2$. Put 
$p=\vartheta(\alpha,n)$.
With high probability,
$\Gnp$ has no terminal $m$-contagious set, with 
$m=\beta\log{n}$, for all $\beta\in[\beta_1,\beta_2]$. 
\end{lem}

\begin{proof}
If $r$-percolation initialized by 
$I\subset[n]$ terminates at size $k$
with $i$ vertices in the top level, then
$I$ is a contagious set for some subgraph $H\subset \Gnp$
of size $k$ with $i$ vertices in the top level, and  
all vertices in $V(H)^c$ are connected to at most $r-1$ vertices 
in $V(H)$.  
Hence the probability that a given $I$ is as such is 
bounded by
\[
{n\choose k-r}m_r(k,i)p^{r(k-r)}
\left(1-{k\choose r}p^r(1-p)^r\right)^{n-k}.
\]
For $k\le \beta_2\log{n}$ and relevant $i$, we have that 
\[
1-{k\choose r}p^r(1-p)^r
= 1-{k\choose r}p^r+O(n^{-1})
\]
where $O(n^{-1})$ depends on $\alpha,\beta_2$, 
but not on $k/\log{n}$ and $i/k$. 
Put $\eps=np^r$. By 
Lemma~\ref{L_mUB} (and the inequalities 
${n\choose k}\le n^k/k!$ and $1-x<e^{-x}$), 
it follows that the expected number of terminal 
$(k,i)$-contagious sets, with $k=\beta\log{n}$ and $i=\gamma k$, 
for some $\beta\le \beta_2$, is bounded (up to a constant) by  
\[
{n\choose r}\left(\frac{\eps k^{r-1}}{(r-1)!}\right)^k
\eps^{-r}
e^{-i-(r-2)k-\eps{k\choose r}}
\lesssim 
n^{\mu_r^*(\alpha,\beta)-\beta\gamma}\log^{r(r-1)}n
\]
where $\lesssim$ denotes inequality 
up to a constant depending on 
$\alpha,\beta_2$, but not on $\beta,\gamma$.

By Lemma~\ref{L_beta*}, we have that 
$\mu_r^*(\alpha,\beta)\le\mu_r^*(\alpha,\beta_1)<0$ 
for all $\beta\in[\beta_1,\beta_2]$. Hence, 
summing over the $O(\log^2n)$ relevant values of $i,k$, 
we find that
the probability that $\Gnp$ contains a terminal 
$m$-contagious set for some $m=\beta\log{n}$, with
$\beta\in[\beta_1,\beta_2]$, is bounded (up to a constant) by 
\[
n^{\mu_r^*(\alpha,\beta_1)}\log^{r(r-1)+2}n
\ll1
\]
as required. 
\end{proof}

\subsection{Almost sure susceptibility }\label{S_rperc}
 
Finally, we complete the proof of Theorem~\ref{T_main}.  Using
Lemmas~\ref{L_mueps},\ref{L_hat:cr2mm},\ref{L_term}, we argue that if $\alpha>\alpha_r$,
then with high probability $\Gnp$ contains a large susceptible subgraph.
By adding independent random graphs with small edge
probabilities, we deduce that percolation occurs with high probability.

\begin{proof}[Proof of Theorem~\ref{T_main}]
  Proposition~\ref{P_SC} gives the sub-critical case 
  $\alpha<\alpha_r$.  Assume
  therefore that $\alpha>\alpha_r$.  Let $\G_*,\G_i$, for $i\ge0$, 
  be independent random graphs with edge
  probabilities $p=\vartheta_r(\alpha_r+\eps,n)$ and
  $p_i=2^{-i(r-1)/r}p_\eps$, where $p_\eps=\vartheta_r(\eps,n)$.  
  Moreover, let $\eps>0$ be sufficiently small so that 
  $\G=\G_*\cup\bigcup_{i\ge0} \G_i$ is a random graph with edge probabilities at
  most $\vartheta_r(\alpha,n)$.
  It thus suffices to show that $\G$ is susceptible. 

\begin{claim}\label{Cl_U0}
  Let $A>0$. 
  With high probability, the graph $\G_*$ contains a
  susceptible subgraph on some set $U_0\subset[n]$ of size
  $|U_0| \geq A \log{n}$.
\end{claim}

\begin{proof}
Using Lemmas~\ref{L_mueps},\ref{L_hat:cr2mm}, 
we show by the second moment method that, 
with high probability,
$\G_*$ contains a susceptible subgraph 
of size at least $(\beta_r^*(\alpha)+\delta_0)\log{n}$, 
for some $\delta_0>0$.
By Lemma~\ref{L_term}, this gives the claim. 

Recall that Lemma~\ref{L_mueps} provides
$\delta,\zeta>0$ so that 
if $k_n/\log{n}
\in[\beta_*(\alpha)+\delta/2,\beta_*(\alpha)+\delta]$,
then $\hat E_r(k_n)\gg n^\zeta$. Fix such a sequence $k_n$.
For each $n$, fix $I_n\subset [n]$ with $|I_n|=r$.
By Lemma~\ref{L_hat:cr2mm}, it follows that 
\begin{align*}
\sum_{I}
\frac{\P(\hat {\cal E}_{k_n}(I)|\hat {\cal E}_{k_n}(I_n))}
{\hat E_r(k_n)}
&\le 1+o(1)
+{n\choose r}^{-1}\sum_{m=1}^{r-1}
{n-m\choose r-m}o((n/k_n)^m)\\
&+n^{-\zeta}\sum_{m=0}^{r-1}
{n-m\choose r-m}\left(
O(k_n^{2r}(k_np)^{r(r-m)})+(k_n/n)^{r-m}
\right)\\
&\le 1+o(1)
+\sum_{m=1}^{r-1}
o((r/k_n)^m)\\
&+n^{-\zeta}\sum_{m=0}^{r-1}
\left(
O(k_n^{2r}((k_np)^rn)^{r-m})+(k_n)^{r-m}
\right)\\
&=1+o(1)+O(n^{-\zeta}\log^{3r}n)\\
&\sim 1
\end{align*}
where the sum is over $I\neq I_n$ 
with $|I|=r$, and $|I\cap I_n|=m$ for some $0\le m<r$.
Hence, by the second moment method,
with high probability some 
$\hat r$-percolation on $\Gnp$ 
grows to size $k_n$ and thus
$\Gnp$ contains a suceptible
subgraph of size $k_n$, as required. 
As discussed, the claim follows by the choice of $k_n$
and Lemma~\ref{L_term}. 
\end{proof}

\begin{claim}
  \label{Cl_sprinkle}
  There is some $A=A(\eps)$ so that if $U_0$ is a set of size 
  $|U_0|\geq A\log n$, then 
  with high probability, $r$-percolation on 
  $\bigcup_{i\ge1}\G_i$ initialized at $U_0$ infects a set
  of vertices of order $n/\log n$.
\end{claim}

\begin{proof}
Let $A=2r(16r/\eps)^{1/(r-1)}$. Moreover assume that $n$ is sufficiently 
large and $\eps$ is sufficiently small so that $A>2$ and 
$A(2^{1-r}\eps/\log{n})^{1/r}<1/2$. 

  We define a sequence of sets $U_i$ as follows.  
  Given $U_i$, we
  consider all vertices not in $U_0,\dots,U_i$, and add to $U_{i+1}$
  some $2^{i+1} A\log n$ vertices 
  that are $r$-connected in $\G_{i+1}$ to $U_i$ (say,
  those of lowest index).

  We first argue that, as long as at most $n/2$
  vertices are included in $\bigcup_{j=1}^iU_j$
  and $2^i<n/\log^2 n$, the probability that
  we can find $2^{i+1} A\log n$ vertices to populate $U_{i+1}$ is at least
  $1-n^{-1}$. Indeed, 
a vertex not in $\bigcup_{j=1}^iU_j$
is at least $r$-connected in 
$\G_{i+1}$ to $U_i$ with probability
bounded from below by  
\[
{|U_i|\choose r}p_{i+1}^r(1-p_{i+1})^{|U_i|-r}
\ge\left(\frac{|U_i| p_{i+1}}{r}\right)^r(1-|U_i|p_{i+1})
\ge
\frac{1}{2}\left(\frac{|U_i| p_{i+1}}{r}\right)^r,
\] 
since, for all large $n$, 
\[
|U_i| p_{i+1}
=2^{-(r-1)/r}(A\log{n})\left(\frac{2^i\eps}{n\log^{r-1}n}\right)^{1/r}
\le A\left(\frac{2^{1-r}\eps}{\log{n}}\right)^{1/r}\le \frac{1}{2}.
\] 
Hence the expected number of such vertices is at least
\[
\frac{n}{2}\frac{1}{2}\left(\frac{|U_i| p_{i+1}}{r}\right)^r
=\frac{\eps}{4r}\left(\frac{A}{2r}\right)^{r-1}(2^iA\log{n})
=2^{i+2}A\log{n}
\]
by the choice of $A$.
Therefore by Chernoff's bound, such a set $U_{i+1}$
of size $2^{i+1}A\log{n}$
can be selected with probability at least 
$1-\exp(-2^{i-1}A\log{n})\le 1-n^{-1}$, since $A>2$ and $i\ge0$, as 
required. 
  
Since the number of levels before reaching $n/2$ vertices is at most
$\log n$, the claim follows. 
\end{proof}

By Claims~\ref{Cl_U0},\ref{Cl_sprinkle}, with high probability, 
$\G_*\cup \bigcup_{i\ge1}\G_i$ contains an $r$-infectious 
subgraph on some $U\subset[n]$ of order  $n/\log{n}$.  
To conclude, we observe that given this, by adding $\G_0$ 
we have that 
$\G=\G_*\cup \bigcup_{i\ge0}\G_i$ is susceptible with high
(conditional) probability. 
Indeed,  
the expected number of vertices 
in $U^c$ 
which are connected in $\G_0$ 
to at most $r-1$ vertices
of $U$ 
is bounded from above by
\[
n\sum_{j=0}^{r-1}
{|U|\choose j}p_0^j(1-p_0)^{|U|-j}
\ll
n(|U|p_0)^r e^{-p_0 (|U|-r)}
\ll n^re^{-n^{(1-1/r)/2}}
\ll1.
\]
Hence $\G$ is susceptible 
with high probability, as required.
\end{proof}

\section{Time dependent branching processes}
\label{S_BP}

In this section, we prove Theorem~\ref{T_BP}, giving estimates for the survival
probabilities for a family of non-homogenous branching process which are
closely related to contagious sets in $\Gnp$.

Recall that in our branching process, the $n$th individual has a Poisson
number of children with mean $\binom{n}{r-1} \eps$.  This does not
specify the order of the individuals, i.e.\ which of these children is
next.  While the order would affect the resulting tree, the choice of
order clearly does not affect the probability of survival.  In light of
this, we can use the breadth first order: Define generation $0$ to be
the first $r-1$ individuals, and let generation $k$ be all children of
individuals from generation $k-1$.  All individuals in a generation
appear in the order before any individual of a later generation.  Let
$Y_t$ be the size of generation $t$, and $S_t = \sum_{i\leq t} Y_i$.

Let $\Psi_r(k,i)$ be the probability that for some $t$ we have $S_t=k$ and
$Y_t=i$.

\begin{lemma}\label{L_BP}
  We have that 
  \[
    \Psi_r(k,i)
    =
    \frac{e^{-\eps{k-i\choose r}}\eps^{k-r}}{(k-r)!}m_r(k,i).
  \]
\end{lemma}

\begin{proof} 
  We first give an equivalent branching process.  Instead of each
  individual having a number of children, children will have $r$
  parents.  We start with $r$ individuals (indexed $0,\dots,r-1$), and
  every subset of size $r$ of the population gives rise to an
  independent $\Poi(\eps)$ additional individuals.  Thus the initial set
  of $r$ individuals produces $\Poi(\eps)$ further individuals, indexed
  $r,\dots$.  Individual $k$ together with each subset of $r-1$ of the
  previous individuals has $\Poi(\eps)$ children, so overall individual
  $k$ has $\Poi\left(\binom{k}{r-1} \eps \right)$ children where $k$ is
  the maximal parent.  

  Let $X_S$ be the number of children of a set $S$ of individuals.  A
  graph contributing to $m_r(k,i)$ requires $\Poi(\eps)$ variables to
  equal $X_S$, so the probability is $\prod e^{-\eps} \eps^{X_s} /
  X_S!$.  Up to generation $t$ this considers $\binom{k-i}{r}$ sets, and
  $\sum X_S = k-r$, giving the terms involving $\eps$ in the claim.  The
  combinatorial terms $\prod X_S!$ and $(k-r)!$ come from possible
  labelings of the graph.
\end{proof}

\begin{proof}[Proof of Theorem~\ref{T_BP}]
Up to the $o(1)$ term appearing in the statement
of the theorem, the survival of $(X_t)$ is equivalent to the 
probability $p_S$ that for some $t$ we have that $S_t\ge k_r$,
where $(S_t)_{t\ge0}$ is as defined above Lemma~\ref{L_BP} and 
$k_r=k_r(\eps)$ is as in the theorem. By Lemma~\ref{L_BP}, 
\[
p_S\ge \sum_i\Psi_r(k_r,i)
\ge \frac{e^{-\eps{k_r\choose r}}\eps^{k_r-r}}{(k_r-r)!}\sum_i m_r(k_r,i)
\ge \frac{e^{-\eps{k_r\choose r}}\eps^{k_r-r}}{(k_r-r)!} m_r(k_r).
\]
By Lemma~\ref{L_hat:mLB}, as $\eps\to0$, the right hand side is 
bounded from below by
\[
e^{-o(k_r)}e^{-(r-2)k_r-\eps{k_r\choose r}}
\left(\eps\frac{k_r^{r-1}}{(r-1)!}\right)^{k_r}
\eps^{-r}
=e^{-\frac{(r-1)^2}{r}k_r(1+o(1))}.
\]

On the other hand, we note that the formula for $\Psi_r(k,i)$ 
in Lemma~\ref{L_BP}
agrees with the upper bound for $P_r(k,i)$ in Lemma~\ref{L_Pki} 
(up to the $1+o(1)$ factor). Hence, 
using the bounds in 
Lemma~\ref{L_mUB} and 
slightly modifying of the proof of Proposition~\ref{P_SC} (since here
we have Poisson random variables instead of Binomial random variables),
it can be shown that  
\[
p_S\le 
e^{o(k_r)}
\frac{e^{-\eps{k_r\choose r}}\eps^{k_r-r}}{(k_r-r)!}m_r(k_r)
=e^{-\frac{(r-1)^2}{r}k_r(1+o(1))}
\]
completing the proof. 
\end{proof}

\section{Graph bootstrap percolation}\label{S_Hr}

Fix $r\ge2$ and a graph $H$. We say that a graph 
$G$ is \emph{$(H,r)$-susceptible} if for some $H'\subset G$
we have that $H'$ is isomorphic to $H$ and $V(H)$ is 
a contagious set for $G$. We call such a subgraph $H'$
a \emph{contagious copy of $H$}. Hence a seed, as
discussed in Section~\ref{S_seeds}, is a contagious clique. 
Let $p_c(n,H,r)$ denote the infimum over $p>0$
such that $\Gnp$ is $(H,r)$-susceptible 
with probability at least $1/2$. 

By the arguments in Sections~\ref{S_pcLB},\ref{S_pcUB}, 
with only minor changes, 
we obtain the following result. We omit the proof.

\begin{thm} \label{T_pcHr}
Fix $r\ge2$ and $H\subset K_r$ with $e(H)=\ell$. Put 
\[
\alpha_{r,\ell}=(r-1)!\left(\frac{(r-1)^2}{r^2-\ell}\right)^{r-1}.
\]
As $n\to\infty$, 
\[
p_c(n,H,r)
=
\left(\frac{\alpha_{r,\ell}}{n\log^{r-1}n}\right)^{1/r}(1+o(1)).
\]
\end{thm}

We obtain Theorem~\ref{T_SP}, from which Theorem~\ref{T_K4} follows,  
as a special case.  

\begin{proof}[Proof of Theorem~\ref{T_SP}]
The result 
follows by Theorem~\ref{T_pcHr},
 taking $r=2$ and $\ell=1$,
in which case $\alpha_{2,1}=1/3$.
\end{proof}

\appendix
\section{Technical lemmas}

We collect in this appendix several technical results used above.

\subsection{Proof of Claim~\ref{C:induct}} 
\label{A:induct}

\begin{proof}[Proof of Claim~\ref{C:induct}]
By the bound 
$i!>\sqrt{2\pi i}(i/e)^i$,
it suffices to verify that
\begin{equation}\label{E_ind}
\frac{(e/i)^i}{\sqrt{2\pi}}\Lambda(i) \le 1
\quad \text{ for $i\ge1$,}
\end{equation}
where $\Lambda(i) = {\rm Li}(-i+1/2,1/e)$ and   
$
{\rm Li}(s,z)=\sum_{j=1}^\infty z^j j^{-s}
$
is the polylogarithm function. 

Let $\Gamma$ denote the gamma function.
>From the relationship 
between ${\rm Li}$ and the Herwitz zeta function, 
it can be shown 
that $\Lambda(i)/\Gamma(i+1/2)\sim1$, as $i\to\infty$, and hence
$(e/i)^i\Lambda(i)\to\sqrt{2\pi}$, as $i\to\infty$. 
It appears (numerically) that $(e/i)^i\Lambda(i)$
increases monotonically to $\sqrt{2\pi}$, however 
this is perhaps not simple to verify (or in fact true). 
Instead, we find a suitable upper bound for $\Lambda(i)$. 

\begin{claim}\label{Cl_Lam}
For all $i\ge1$, we have that 
\[
\Lambda(i)
<
\Gamma(i+1/2)(1+ab^i)
\]
where $a=\zeta(3/2)$ and $b=e/(2\pi)$,
and $\zeta$ is the 
Riemann zeta function.
\end{claim}

\begin{proof}
For all $|u|<2\pi$ and 
$s\notin\NN$,  
we have the  
series representation 
\[
{\rm Li}(s,e^u)
=\Gamma(1-s)(-u)^{s-1}
+\sum_{\ell=0}^\infty\frac{\zeta(s-\ell)}{\ell!}u^\ell.
\]
Hence
\begin{equation}\label{E_LGZ}
\Lambda(i)
=\Gamma(i+1/2)
+\sum_{\ell=0}^\infty \frac{(-1)^\ell}{\ell!}\zeta(1/2-i-\ell).
\end{equation}
Recall the functional equation for $\zeta$,
\[
\zeta(x)=2^x\pi^{x-1}\sin(\pi x/2)\Gamma(1-x)\zeta(1-x).
\]
Therefore, since $\zeta(1/2+x)>0$ is decreasing in $x\ge1$
we have that, for all relevant $i,\ell$,
\begin{equation}\label{E_ZG}
|\zeta(1/2-i-\ell)|
\le a\sqrt{\frac{2}{\pi}}\frac{\Gamma(\ell+i+1/2)}{(2\pi)^{\ell+i}}
<a\frac{\Gamma(\ell+i+1/2)}{(2\pi)^{\ell+i}}.
\end{equation}

Applying \eqref{E_LGZ},\eqref{E_ZG}
(and the inequalities 
$\Gamma(x+\ell)<(x+\ell-1)^\ell\Gamma(x)$, 
$\ell!>\sqrt{2\pi \ell}(\ell/e)^\ell$,
and $(1+x/\ell)^\ell<e^\ell$),
we find that, for all $i\ge1$, 
\begin{align*}
\frac{\Lambda(i)}{\Gamma(i+1/2)}-1
&< 
\frac{a}{(2\pi)^i}
\sum_{\ell=0}^\infty \frac{(\ell+i-1/2)^\ell}{(2\pi)^{\ell}\ell!}\\
&< 
\frac{ab^i}{e^i}
\left(1+
\sum_{\ell=1}^\infty 
\frac{1}{\sqrt{2\pi \ell}}
\left(\frac{e}{2\pi}\left(1+\frac{i-1/2}{\ell}\right)\right)^\ell
\right)\\
&< 
ab^i
\left(\frac{1}{e}+
\frac{1}{\sqrt{2e\pi}}
\sum_{\ell=1}^\infty 
\left(\frac{e}{2\pi}\right)^\ell
\right)\\
&<
ab^i
\end{align*}
establishing the claim.
\end{proof}

By Claim~\ref{Cl_Lam}, the formula 
\[
\Gamma(i+1/2)
=\sqrt{\pi}\frac{i!}{4^i}{2i\choose i}, 
\]
and the bounds 
\[
{2i\choose i}<\frac{4^i}{\sqrt{\pi i}}\left(1-\frac{1}{9i}\right)
\]
and 
\[
i!
<\sqrt{2\pi i} \left(\frac{i}{e}\right)^i
\left(1-\frac{1}{12i}\right)^{-1}
\]
(valid for all $i\ge1$), we find that
\begin{equation}\label{E_indUB}
\frac{(e/i)^i}{\sqrt{2\pi}}\Lambda(i)
<\frac{4}{3}\frac{9i-1}{12i-1}
(1+ab^i)
\quad \text{for $i\ge1$}.
\end{equation}
Differentiating the right hand side of \eqref{E_indUB},
and dividing by the positive term
$4/(3(12i-1)^2)$, we obtain 
\[
3+ab^i\left(3
+\log(b)(108i^2-12i+1)\right)
\]
which, for $i\ge11$, is bounded from below by 
\[
3+108ab^i\log(b)i^2
>3-237 b^ii^2
>0.
\]
Hence, for $i\ge11$,  the right hand side of \eqref{E_indUB}
increases monotonically to $1$ as $i\to\infty$. 
It follows that 
\eqref{E_ind} holds for all $i\ge11$. 
Inequality \eqref{E_ind}, for $i\le10$, can be 
verified numerically (e.g., by interval arithmetic), 
completing the proof of Claim~\ref{C:induct}.
\end{proof}

\subsection{Proof of Claim~\ref{C_smallbeta}} 
\label{A:C_smallbeta}

\begin{proof}[Proof of Claim~\ref{C_smallbeta}]
By \eqref{E_mu}, we have that 
\[
\frac{\partial^2}{\partial\gamma^2}\mu_r(\alpha,\beta,\gamma)
=
-\frac{\alpha\beta^r}{(r-2)!}(1-\gamma)^{r-2}<0. 
\]
The result thus follows, noting that 
\[
\frac{\partial}{\partial\gamma}\mu_r(\alpha,\beta,\gamma)
=
-\beta\left(1-\frac{\alpha\beta^{r-1}}{(r-1)!}(1-\gamma)^{r-1}\right) 
\]
and hence for any $\xi<1$ and $\gamma\in(0,1)$, 
\[
\frac{\partial}{\partial\gamma}
\mu_r(\alpha,\xi\beta_r(\alpha),\gamma)
=
-\xi\beta_r(\alpha)
\left(1-(\xi(1-\gamma))^{r-1}\right)
<0.
\qedhere
\]
\end{proof}

\subsection{Proof of Claim~\ref{C:Epmax}} 
\label{A:Epmax}

\begin{proof}[Proof of Claim~\ref{C:Epmax}]
By Lemma~\ref{L_Eki}, 
for all  $k=\beta\log{n}$ and $i=\gamma k$ as in the lemma, 
we have that 
\begin{equation}\label{E_Ep,p=1}
E_{r}(k,i)\lesssim n^{\mu_r(\alpha,\beta,\gamma)}\log^{r(r-1)}n.
\end{equation}

We find a suitable upper bound for $p_{r}(k,i)$ as follows.
For $\beta<\beta_r(\alpha)$, put    
$\ell_\beta=\xi_\beta\log{n}$, where 
$\xi_{\beta}=\beta_r(\alpha)  -\beta$.
For a given set $V$ of size $k$ with
$i$ vertices identified as the top level, there are $a_r(k,i)$ ways
to select $r$ vertices in $V$ with at least one in the top level. 
Hence, for $k=\beta\log{n}$ with $\beta<\beta_r(\alpha)$, it
follows that 
\[
p_{r}(k,i)
\le 
{n\choose \ell_\beta}(a_r(k,i)p^r)^{\ell_\beta}.
\]
By Claim~\ref{C_a}, we have that 
$a_r(k,i)<ik^{r-1}/(r-1)!$.
Hence, applying the bound ${n\choose \ell}\le (ne/\ell)^\ell$,
we find that  
\[
p_{r}(k,i)
\le 
\left(\frac{e\alpha\beta^r\gamma}{\xi_\beta(r-1)!}\right)^{\ell_\beta}.
\]
Hence, by Lemma~\ref{L_Eki},  
\begin{equation}\label{E_EpUB}
E_{r}(k,i)p_{r}(k,i)
\lesssim
n^{\bar \mu_r(\alpha,\beta,\gamma)}
\log^{r(r-1)}n
\end{equation}
where
\begin{equation}\label{E_barmu}
\bar \mu_r(\alpha,\beta,\gamma)
=
\mu_r(\alpha,\beta,\gamma)
+\xi_\beta\log\left(\frac{e\alpha\beta^r\gamma}{\xi_\beta(r-1)!}\right).
\end{equation}

Therefore, by \eqref{E_Ep,p=1},\eqref{E_EpUB},
we obtain Claim~\ref{C:Epmax} by the following fact. 

\begin{claim}\label{C_barmu}
For any $\gamma\in(0,1)$, we have that 
\[
\min\{
\mu_r(\alpha,\beta,\gamma),\bar \mu_r(\alpha,\beta,\gamma)
\}\le \mu_r^*(\alpha,\beta_r)
\]
for all $\beta\in(0,\beta_r(\alpha)]$. 
\end{claim}

\begin{proof}
For convenience, we simplify notations as follows.   
Put $\beta_r=\beta_r(\alpha)$. We parametrize $\beta$
using a variable $\delta$: for $\delta\in(0,1]$, let $\beta_\delta=\delta\beta_r$. 
For $\gamma\in(0,1)$, let 
$\mu_r(\delta,\gamma)=\mu_r(\alpha,\beta_\delta,\gamma)$,
$\bar\mu_r(\delta,\gamma)=\bar \mu_r(\alpha,\beta_\delta,\gamma)$, 
and 
$\delta_\gamma=\delta_\gamma(r)=1-\sqrt{\gamma/r}$. 
Finally, put $\mu_r^*=\mu_r(1,0)=\mu_r^*(\alpha,\beta_r)$.
In this notation, Claim~\ref{C_barmu} states
that 
\[
\min\{\mu_r(\delta,\gamma),\bar \mu_r(\delta,\gamma)\}
\le
\mu_r^*,
\quad\text{ for $\delta\in(0,1]$.}
\]

Since $\alpha\beta_r^{r-1}/(r-1)!=1$,
it follows that $\alpha\beta_\delta^{r-1}/(r-1)!=\delta^{r-1}$.
Therefore, by \eqref{E_mu},\eqref{E_barmu},  
we have that 
\begin{equation}\label{E_mu-dg}
\mu_r(\delta,\gamma)
=
r-\beta_r\left(\frac{\delta^r}{r}(1-\gamma)^r
+\delta(r-2+\gamma)
-(r-1)\delta\log\delta\right)
\end{equation}
and
\begin{equation}\label{E_barmu-dg}
\bar\mu_r(\delta,\gamma)
=
\mu_r(\delta,\gamma)
+\beta_r(1-\delta)\log\left(\frac{e\gamma\delta^r}{1-\delta}\right).
\end{equation}

We obtain Claim~\ref{C_barmu} by the 
following subclaims (as we explain below the statements). 

\begin{subclaim}\label{SC_mumubar}
For any fixed $\gamma\in(0,1)$, 
we have that 
$\mu_r(\delta,\gamma)$ and $\bar\mu_r(\delta,\gamma)$ are
convex and concave in $\delta\in(0,1)$, respectively. 
\end{subclaim}

\begin{subclaim}\label{SC_mudelgam}
For $\gamma\in(0,1)$, 
we have that 
\begin{enumerate}[nolistsep,label={(\roman*)}] 
\item $\mu_r(1,\gamma)<\mu_r^*$,
\item $\mu_r(\delta_\gamma,\gamma)<\mu_r^*$, and 
\item $e\gamma\delta_\gamma^r/(1-\delta_\gamma)<1$.
\end{enumerate}
\end{subclaim}
 
Indeed, by Sub-claim~\ref{SC_mudelgam}(ii),(iii), 
we have that 
$\bar \mu_r(\delta_\gamma,\gamma)
<\mu_r(\delta_\gamma,\gamma)<\mu_r^*$. 
Therefore, noting that  
$\lim_{\delta\to1^-}\bar\mu_r(\delta,\gamma)=\mu_r(1,\gamma)$,
$\lim_{\delta\to0^+}\mu_r(\delta,\gamma)=r$, and
$\lim_{\delta\to0^+}\bar\mu_r(\delta,\gamma)=-\infty$
(see \eqref{E_mu-dg},\eqref{E_barmu-dg}),
we then obtain 
Claim~\ref{C_barmu} by applying 
Sub-claims~\ref{SC_mumubar},\ref{SC_mudelgam}(i).

\begin{proof}[Proof of Sub-claim~\ref{SC_mumubar}]
By \eqref{E_mu-dg}, for any $\gamma\in(0,1)$,
we have that 
\[
\frac{\partial^2}{\partial\delta^2}\mu_r(\delta,\gamma)
=\frac{(r-1)\beta_r}{\delta}(1-\delta^{r-1}(1-\gamma)^r)>0
\]
for all $\delta\in(0,1)$. Moreover,
by \eqref{E_mu-dg},\eqref{E_barmu-dg}, 
the above expression, 
and noting that 
\[
\frac{\partial^2}{\partial\delta^2}
(1-\delta)\log\left(\frac{e\gamma\delta^r}{1-\delta}\right)
=
-\frac{r-(r-1)\delta^2}{\delta^2(1-\delta)},
\]
it follows that, for any $\gamma\in(0,1)$, 
\begin{align*}
\frac{\partial^2}{\partial\delta^2}\bar\mu_r(\delta,\gamma)
&=-\frac{\beta_r}{\delta^2(1-\delta)}
\left(
r-(r-1)\delta^2-\delta(1-\delta)(1-\delta^{r-1}(1-\gamma)^r)
\right)\\
&=-\frac{\beta_r}{\delta^2(1-\delta)}
\left(1+(r-1)(1-\delta)(1+\delta^r(1-\gamma)^r)\right)\\
&<0
\end{align*}
for all $\delta\in(0,1)$.
The claim follows. 
\end{proof}

\begin{proof}[Proof of Sub-claim~\ref{SC_mudelgam}]
Since, by \eqref{E_mu-dg},
\[
\mu_r(1,\gamma)
=
r-\beta_r\left(\frac{(1-\gamma)^r}{r}+r-2+\gamma\right)
\]
claim (i) follows immediately 
by the inequality $(1-x)^y\ge1-xy$.

Next, we note that, by \eqref{E_mu-dg} 
and the bound $\log{x}<(x^2-1)/(2x)$, 
\[
\mu_r^*-\mu_r(\delta,\gamma)
>\beta_r\left(
\frac{\delta^r}{r}(1-\gamma)^r+\delta(r-2+\gamma)+\frac{r-1}{2}(1-\delta^2)
-\frac{(r-1)^2}{r}
\right).
\]
Hence, 
to establish claim (ii), 
if suffices to verify that $f_r(\delta_\gamma,\gamma)>(r-1)^2/r$,
where
\[
f_r(\delta,\gamma)
=\frac{\delta^r}{r}(1-\gamma)^r+\delta(r-2+\gamma)+\frac{r-1}{2}(1-\delta^2).
\]
The case $r=2$ is straightforward, since in this case 
\[
f_2(\delta_\gamma,\gamma)
=
f_2(1-\sqrt{\gamma/2},\gamma)
=
\frac{1}{2}\left(1+\sqrt{2}\gamma^{3/2}(1-\gamma)+\frac{\gamma^3}{2}\right)
>\frac{1}{2}
\]
for all $\gamma\in(0,1)$. 
For the remaining cases $r>2$, we show that 
$f_r(\delta_\gamma,\gamma)$ is increasing 
in $\gamma$. 
Since 
$\lim_{\gamma\to0^+}f_r(\delta_\gamma,\gamma)=(r-1)^2/r$ 
this implies the claim.  
To this end, we note that 
\[
\frac{\partial}{\partial\delta}f_r(\delta,\gamma)
=\delta^{r-1}(1-\gamma)^r+r-2+\gamma-(r-1)\delta,
\]
\[
\frac{\partial}{\partial\gamma}\delta_\gamma
=-\frac{1}{2\sqrt{\gamma r}},
\]
(recalling that $\delta_\gamma=1-\sqrt{\gamma/r}$) and
\[
\frac{\partial}{\partial\gamma}f_r(\delta,\gamma)
=\delta-\delta^r(1-\gamma)^{r-1}.
\]
Hence, differentiating $f_r(\delta_\gamma,\gamma)$ with
respect to $\gamma$, we obtain 
\begin{equation}\label{E_dgam:fr}
\delta_\gamma
-\delta_\gamma^{r-1}(1-\gamma)^{r-1}
\left(\frac{1-\gamma}{2\sqrt{\gamma r}}+\delta_\gamma
\right)
-\frac{r-2+\gamma-(r-1)\delta_\gamma}{2\sqrt{\gamma r}}.
\end{equation}
By the inequality $(1-x)^y<1/(1+xy)$, 
\[
\delta_\gamma^{r-1}(1-\gamma)^{r-1}
<
\frac{1}{(1+(r-1)\gamma)(\delta_\gamma+\sqrt{\gamma r})}.
\]
Therefore the expression at \eqref{E_dgam:fr} multiplied 
by 
\[
2\sqrt{\gamma r}
(1+(r-1)\gamma)
(\delta_\gamma+\sqrt{\gamma r})>0
\]
is bounded from above by 
\[
(1+(r-1)\gamma)
(\delta_\gamma+\sqrt{\gamma r})
(1-\gamma+2\sqrt{\gamma r}\delta_\gamma
-(r-1)(1-\delta_\gamma))
-(1-\gamma+2\sqrt{\gamma r}\delta_\gamma).
\]
which after some straightforward manipulations
reduces to 
\[
\frac{(r-1)\gamma}{r}
\left(\sqrt{\gamma r}(2r-3((r-1)\gamma+1))
+(r^2-3r-1)\gamma +2r +1
\right).
\]

Put 
\[
F_r(\gamma)=\sqrt{\gamma r}(3\gamma+2r-3-3r\gamma)
+r^2\gamma-3r\gamma+2r+1-\gamma.
\]
We claim that $F_r(\gamma)>0$ for all $r>2$ and
$\gamma\in(0,1)$. Note that this implies that 
the expression at \eqref{E_dgam:fr} is positive
for all such $r,\gamma$, and hence that 
$f_r(\delta_\gamma,\gamma)$ is increasing
in $\gamma$ for any such $r$, as desired.
To see this, 
we observe that 
\[
\frac{\partial^2}{\partial\gamma^2}F_r(\gamma)
=-\frac{1}{4}\sqrt{\frac{r}{\gamma^3}}(2r-3+9(r-1)\gamma)<0,
\]
$\lim_{\gamma\to0^+} F_r(\gamma)=2r+1>0$, and 
$\lim_{\gamma\to1^-} F_r(\gamma)=r(r-\sqrt{r}-1)>0$
for all $r>2$ and $\gamma\in(0,1)$. 
Altogether, we have established claim (ii) for all $r\ge2$.

Finally, for claim (iii), let 
$g_r(\delta,\gamma)=e\gamma\delta^r/(1-\delta)$. 
In this notation, claim (iii) states that 
$g_r(\delta_\gamma,\gamma)<1$. 
To verify this inequality, we note that 
\[
\frac{\partial}{\partial\delta}
g_r(\delta,\gamma)
=
\frac{e\gamma\delta^{r-1}}{(1-\delta)^2}(r-(r-1)\delta)
\]
and hence
\[
\frac{\partial}{\partial\delta}
g_r(\delta_\gamma,\gamma)
=
e\delta_\gamma^{r-1}(r+(r-1)\sqrt{\gamma r}).
\]
Therefore, noting that
\[
\frac{\partial}{\partial\gamma}
g_r(\delta,\gamma)\mid_{\delta=\delta_\gamma}
=
\frac{e\delta_\gamma^r}{1-\delta_\gamma}
=e\delta_\gamma^{r-1}\left(\sqrt{\frac{r}{\gamma}}-1\right)
\]
and recalling that
\[
\frac{\partial}{\partial\gamma}\delta_\gamma
=-\frac{1}{2\sqrt{\gamma r}}
\]
it follows that 
\[
\frac{\partial}{\partial\gamma}
g_r(\delta_\gamma,\gamma)
=
\frac{e\delta_\gamma^{r-1}}{2}
\left(\sqrt{\frac{r}{\gamma}}
-(r+1)
\right).
\]
Therefore, for any $r\ge2$, $g_r(\delta_\gamma,\gamma)$
is maximized at $\gamma=r/(r+1)^2$. 
By the inequality $(1-x/n)^n<e^{-x}$, we find that 
\[
g_r(r/(r+1)^2)
=
\frac{er}{r+1}\left(1-\frac{1}{r+1}\right)^r<
\frac{r}{r+1}\left(1-\frac{1}{r+1}\right)^{-1}
=1
\] 
giving the claim. 
\end{proof}

As discussed,  
Sub-claims~\ref{SC_mumubar},\ref{SC_mudelgam} imply 
Claim~\ref{C_barmu}. 
\end{proof}

To conclude, we recall that Claim~\ref{C_barmu} implies
Claim~\ref{C:Epmax}.
\end{proof}

\subsection{Proof of Claim~\ref{Cl_hat:rhoind}} 
\label{A_Cl_hat:rhoind}

\begin{proof}[Proof of Claim~\ref{Cl_hat:rhoind}]
We recall the relevant quantities defined in the proof 
of Lemma~\ref{L_hat:mkiLB}, 
see \eqref{E_hat:rho},\eqref{E_hat:rhoLB-smalli},\eqref{E_hrk}.
We have that 
\[
\hat \rho_r(k,i)\ge 
\xi' e^{-(r-2)k-h_r(k)}\quad \text{for $k>k_r$ and $i\le k/\log^2{n}$}
\]
where 
\[
h_r(k)=f_r(k)
-\log g_r\left(k,\frac{k}{\log^2{k}}\right)
+\frac{k}{\log^2{k}}\log\left(\frac{k}{\log^2{k}}\right),
\]
$f_r(k)$ is non-decreasing and $f_r(k)\ll k$, and $g_r(k,i)=e^{(r-2)i}\left(\frac{k-i}{k}\right)^{(r-2)k}$.
Claim~\ref{Cl_hat:rhoind} states that 
for some $\xi>0$,  for all large $k$ and   
$i\le (\eps/r)^2k$,  we have that 
$\hat \rho_r(k,i) \ge \xi e^{-i\eps-(r-2)k-h_r(k)}$. 

\begin{subclaim}\label{SC_hrk}
For all $k>k_r$, we have that 
$h_r(k)$ is increasing
in $k$. 
\end{subclaim}

\begin{proof}
Since $f_r(k)$ is non-decreasing 
and $k/\log^2k$ is increasing, it suffices by 
\eqref{E_hrk} to show that 
$g_r(k,k/\log^2k)$ is decreasing for $k>k_r$
(and assuming $r>2$, as else $g_r\equiv1$
and so there is nothing to prove). 
To this end, we note that
\[
\frac{\partial}{\partial i}g_r(k,i)
=-\frac{(r-2)i}{k-i}g_r(k,i),
\]
\[
\frac{\partial}{\partial k}\frac{k}{\log^2k}
=\frac{\log{k}-2}{\log^3k},
\]
and
\[
\frac{\partial}{\partial k}g_r(k,i)
=
\frac{r-2}{k-i}\left((k-i)\log\left(\frac{k-i}{k}\right)+i\right)g_r(k,i)
\]
Hence, differentiating $g_r(k,k/\log^2k)$ with respect to 
$k$, and
dividing by 
\[
-\frac{(r-2)k}{k(1-\log^{-2}k)\log^3k}g_r(k,k/\log^2k)
<0
\]
we obtain
\[
(\log^3{k})(1-\log^{-2}{k})\log\left(\frac{\log^2{k}}{\log^2{k}-1}\right)
-\frac{\log^3{k}-\log{k}+2}{\log^2{k}}.
\]
By the inequality $\log{x}>2(x-1)/(x+1)$ (valid
for $x>1$), the above expression is 
bounded from below by 
\[
\frac{\log^3{k}-4\log^2{k}-\log{k}+2}
{(\log^2k)(2\log^2{k}-1)}
>
\frac{\log{k}-5}
{2\log^2{k}-1}>0
\]
for all $k>k_r$, since
$k_r>e^{r/\eps}>e^{r(r+1)}$ and $r>2$.
The claim follows. 
\end{proof}

By Sub-claim~\ref{SC_hrk}, fix some $k_*=k_*(r,\eps)>k_r$ so that 
$k/\log^2k$ is larger than $9(r/\eps)^4$ and $(r+2)!/(1-\eps)$ 
for all $k\ge k_*$,
and $h_r(k)$ is increasing for all 
$k\ge (1-(\eps/r)^2)k_*$. 
By 
\eqref{E_hat:rhoLB-smalli}, 
select some $\xi(r,\eps)\le \xi'$ so that 
the claim holds for all $k>k_r$ and relevant $i$, provided either 
$i\le k/\log^2k$ or $k\le k_*$. 
 
We establish the remaining cases, 
$k>k_*$ and $k/\log^2k<i\le (\eps/r)^2$, 
by induction. To this end, let $k>k_*$ be given, 
and assume that the claim
holds for all $k'<k$ and relevant $i$. 
By \eqref{E_hat:mrec} it follows that 
\begin{equation}\label{E_hat:rhorec}
\hat \rho_r(k,i)
\ge
\sum_{j=1}^{k-r-i} 
\hat B_r(k,i,j) 
\hat \rho_r(k-i,j)
\quad (i< k-r)
\end{equation}
where 
\[
\hat B_r(k,i,j) 
=\frac{j^i}{i!}
\left(\frac{k-i}{k}\right)^{(r-2)k}
\left(\frac{k-i-j}{k-i}\right)^{k-i}
\left(\frac{(r-1)!}{(k-i)^{r-1}}\frac{\hat a_r(k-i,j)}{j}\right)^i.
\]

\begin{subclaim}\label{SC_hat:a}
For all $(r+2)!\le i,j\le k/r^2$, we have that  
\[
\hat B_r(k,i,j) \ge
\frac{j^i}{i!}\left(\frac{k-i}{k}\right)^{(r-2)k}
\left(\frac{k-i-j}{k-i}\right)^{k+(r-2)i}.
\]
\end{subclaim}

\begin{proof}
By the formula for $\hat B_r(k,i,j)$ above, it suffices to 
show that 
\[
\frac{(r-1)!}{(k-i)^{r-1}}\frac{\hat a_r(k-i,j)}{j}
>
\left(\frac{k-i-j}{k-i}\right)^{r-1}.
\]
To this end, we note that by \eqref{E_hat:a} and Claim~\ref{C_a}
the left hand side is bounded from below by   
\[
\frac{1}{j}\sum_{\ell=1}^{j}\left(\frac{k-i-\ell}{k-i}\right)^{r-1}
-\frac{2r!}{k-i}
\]
Since, for any integer $m$, 
$(1-y/x)^m-(1-(y+1/2)/x)^m$ is 
decreasing in $y$, for $y<x$,  
it follows that  
\[
\frac{1}{j}\sum_{\ell=1}^{j}\left(\frac{k-i-\ell}{k-i}\right)^{r-1}
\ge
\left(\frac{k-i-(j+1)/2}{k-i}\right)^{r-1}.
\]
Thus, applying the inequalities  
$1-xy\le (1-x)^y\le1/(1+xy)$, we find that  
\[
\frac{(r-1)!}{(k-i)^{r-1}}\frac{\hat a_r(k-i,j)}{j}
-
\left(\frac{k-i-j}{k-i}\right)^{r-1}
\]
is bounded from below by
\[
1-\frac{(j+1)(r-1)}{2(k-i)}-\frac{2r!}{k-i}
-\frac{1}{1+j(r-1)/(k-i)}
\]
which equals
\[
\frac{((r-1)j-(r+4r!-1))(k-i)-((r-1)j+(r+4r!-1))(r-1)j}
{2(k-i)(k-i+(r-1)j)}.
\]
It thus remains to show that the numerator
in the above expression is non-negative, for
all $i,j$ as in the claim. 
To see this, we  
observe that $r+4r!-1<(r-1)(r+2)!$ for all $r\ge2$. 
Hence, 
for $(r+2)!\le i,j\le k/r^2$ and $r\ge2$, 
the numerator 
divided by $(r-1)k>0$
is bounded from below by  
\[
(j-(r+2)!)\left(1-\frac{1}{r^2}\right)
-(j+(r+2)!)\frac{1}{r^2}
=
\left(1-\frac{2}{r^2}\right)
(j-(r+2)!)\ge0
\]
as required. The claim follows. 
\end{proof}

Applying Sub-claim~\ref{SC_hat:a}, the inductive hypothesis, 
and the bound $i!<3\sqrt{i}(i/e)^i$ to \eqref{E_hat:rhorec},  
it follows that 

\begin{equation}\label{E_hat:rhoind}
\hat \rho_r(k,i)
>
\xi\frac{e^{-(r-2)k+(r-1)i-h_r(k-i)}}{3\sqrt{i}}
\left(\frac{k-i}{k}\right)^{(r-2)k}
\sum_{j\in{\cal J}_{r,\eps}}
\psi_{r,\eps}(i/k,j/i)^k
\end{equation}
where 
${\cal J}_{r,\eps}(k,i)$ 
is the set of $j$ satisfying 
$(r+2)!\le j\le (\eps/r)^2(k-i)$, and 
\[
\psi_{r,\eps}(\gamma,\delta)
=
\delta^\gamma
e^{-\delta\gamma\eps}
\left(1-\frac{\delta\gamma}{1-\gamma}\right)^{1+\gamma(r-2)}.
\]
 
\begin{subclaim}\label{SC_dels}
Put $\delta_\eps=1-\eps$ and 
$\delta_{r,\eps}=\delta_\eps+(\eps/r)^2$.
For any fixed $\gamma\le (\eps/r)^2$, we have that 
$\psi_{r,\eps}(\gamma,\delta)$ is increasing in $\delta$,
for $\delta\in[\delta_\eps,\delta_{r,\eps}]$. 
\end{subclaim}

\begin{proof}
Differentiating $\psi_{r,\eps}(\gamma,\delta)$ with respect to $\delta$, we obtain
\[
\frac{\psi_{r,\eps}(\gamma,\delta)\gamma}
{\delta(1-\gamma-\delta\gamma)}
\left(
\eps\gamma\delta^2
-(1+\eps+\gamma(r-1-\eps))\delta
+1-\gamma
\right).
\]
Hence, to establish the claim, it suffices to 
show that 
\[
\eps\gamma\delta_{r,\eps}^2
-(1+\eps+\gamma(r-1-\eps))\delta_{r,\eps}
+1-\gamma
\]
is positive for relevant $\gamma$. 
Moreover, since the above expression is decreasing in 
$\gamma$, we need only verify the 
case $\gamma=(\eps/r)^2$. Setting $\gamma$
as such in the above expression, and then dividing 
by $\eps^2/r^6$, we obtain 
\[
r^6
-(1-\eps)r^5
-(1+3\eps^2-\eps^3)r^4
-r^3\eps^2
+\eps^2(1+3\eps-2\eps^2)r^2
+\eps^5.
\]
For $\eps<1/r$ and $r\ge2$, this expression is bounded 
from below by
\[
r(r^5-r^4-(1+3/r^2)r^3-1)\ge r>0
\]
as required, giving the claim.
\end{proof}

By the choice of $k_*$ and since $k>k_*$, 
for all relevant $k/\log^2k\le i\le (\eps/r)^2k$, we have that 
$\delta_\eps i\ge(r+2)!$ and 
\[
\frac{\delta_{r,\eps} i}{k-i}
\le
(\eps/r)^2\frac{1-\eps+(\eps/r)^2}{1-(\eps/r)^2}
\le
(\eps/r)^2
\]
where the second inequality follows since
\[
\frac{\partial}{\partial\eps}\frac{1-\eps+(\eps/r)^2}{1-(\eps/r)^2}
=-r^2\frac{(r^2+\eps^2-4\eps)}{(r-\eps)^2(r+\eps)^2}<0
\]
for all $r\ge2$. Hence, for all such $i,k$, 
we have that 
$j\in{\cal J}_{r,\eps}(k,i)$ for all $j\in[\delta_\eps,\delta_{r,\eps}]$.
Therefore, for any such $i,k$, by \eqref{E_hat:rhoind}
and Sub-claim~\ref{SC_dels}, we have that 
\begin{align*}
\hat \rho_r(k,i)
&>
\xi
\frac{e^{-(r-2)k+(r-1)i-h_r(k-i)}}{3\sqrt{i}}
\left(\frac{k-i}{k}\right)^{(r-2)k}
\sum_{\delta_\eps i\le j\le \delta_{r,\eps}i}\psi_{r,\eps}(i/k,j/i)^k\\
&>
\xi
\frac{(\delta_{r,\eps}-\delta_\eps)\sqrt{i}}{3}
e^{-(r-2)k+(r-1)i-h_r(k-i)}
\left(\frac{k-i}{k}\right)^{(r-2)k}
\psi_{r,\eps}(i/k,\delta_\eps)^k\\
&>
\xi
e^{-(r-2)k+(r-1)i-h_r(k-i)}
\left(\frac{k-i}{k}\right)^{(r-2)k}
\psi_{r,\eps}(i/k,\delta_\eps)^k
\end{align*}
where the last inequality follows since
for any such $i,k$, by the choice of $k_*$ and since $k>k_*$, 
we have that 
$\delta_{r,\eps}-\delta_\eps=(\eps/r)^2>3/\sqrt{i}$.  

\begin{subclaim}\label{SC_zeta}
Fix $k/\log^2k\le i\le (\eps/r)^2k$, and define 
$\zeta_r(k,i)$ such that 
\[
\hat\rho_r(k,i)=\xi e^{-\zeta_r(k,i)\eps i-(r-2)k-h_r(k)}.
\] 
We have that $\zeta_r(k,i)<1$. 
\end{subclaim}

\begin{proof} 
Letting $\gamma=i/k$, it follows by the bound for $\hat \rho_r(k,i)$ 
above,
and since $k>k_*$ and hence 
$h_r(k-i)<h_r(k)$ by the choice of $k_*$, 
that $\zeta_r(k,i)$ is bounded from above by 
\[
\delta_\eps
-\frac{r-1}{\eps}
-\frac{r-2}{\eps\gamma}\log(1-\gamma)
-\frac{1}{\eps}\log{\delta_\eps}
-\frac{1+\gamma(r-2)}{\eps\gamma}
\log\left(1-\frac{\delta_\eps\gamma}{1-\gamma}\right).
\]

Recall that $\delta_\eps=1-\eps$.
Applying the bound $-\log(1-x)\le x/(1-x)$ for $x=\gamma$ and 
$x=\delta_\eps\gamma/(1-\gamma)$,
and the bound $-\log(1-x)\le x+(1+x)x^2/2$  
for $x=\eps$ (valid for any $x<1/3$,
and so for all relevant $\eps<1/(r+1)$ with $r\ge2$), 
we find that 
the expression above is bounded from above by 
\[
\nu(\eps,\gamma) 
=
2-\frac{\eps(1-\eps)}{2}
-\frac{1-(r-1)\gamma}{\eps(1-\gamma)}
+\frac{(1-\eps)(1+(r-2)\gamma)}{\eps(1-(2-\eps)\gamma)}.
\]

Therefore, noting that
\[
\frac{\partial}{\partial\gamma}\nu(\eps,\gamma) 
=
\frac{r-2}{\eps(1-\gamma)^2}
+\frac{(1-\eps)(r-\eps)}{\eps(1-(2-\eps)\gamma)^2}
>0,
\]
to establish the subclaim, it suffices to verify that 
$\nu(\eps,(\eps/r)^2)<1$ for all $r\ge2$ and $\eps<1/(r+1)$. 
Furthermore, 
since
\[
\nu(\eps,(\eps/r)^2) 
=
2-\frac{\eps(1-\eps)}{2}
-\frac{r^2-\eps^2(r-1)}{\eps(r^2-\eps^2)}
+\frac{(1-\eps)(r^2+\eps^2(r-2))}{\eps(r^2-2\eps^2+\eps^3)}
\]
and hence
\[
\frac{\partial}{\partial r}\nu(\eps,(\eps/r)^2) 
=
-\frac{\eps(r(r-4)+\eps^2)}{(r^2-\eps^2)^2}
-\frac{\eps(1-\eps)(r(r-2\eps)+\eps^2(2-\eps))}
{(r^2-2\eps^2+\eps^3)^2}<0
\]
for all $k\ge4$ and $\eps<1$, we need only verify 
the cases $r\le4$. 

To this end, let $\eta(r,\eps)$ denote the 
difference of the numerator and denominator 
of $\nu(\eps,(\eps/r)^2)$ (in its factorized form), namely
\begin{multline*}
-\eps^7
+3\eps^6
+(r^2-4)\eps^5
-2(2r^2-2r+1)\eps^4
+(5r^2-6r+8)\eps^3\\
+r^2(r^2-2r-2)\eps^2
-r^2(r-2)^2\eps.
\end{multline*} 
For all $\eps<1/3$, we have that 
\[
\eta(2,\eps)
=-\eps^2(1-\eps)(2-\eps)(2+\eps)(2-2\eps+\eps^2)
<-\eps^2
<0.
\]
Similarly, 
\[
\eta(3,\eps)
=
-\eps(9-9\eps-35\eps^2+26\eps^3-5\eps^4-3\eps^5+\eps^6)
<-\eps
<0
\]
and
\[
\eta(4,\eps)
=
-\eps(64-96\eps-64\eps^2+50\eps^3-12\eps^4-3\eps^5+\eps^6)
<
-\eps
<0.
\]

It follows that $\nu(\eps,(\eps/r)^2)<1$ 
for all $\eps<1/3$ and $k\le4$,
and hence for all $k\ge2$, giving the subclaim.
\end{proof}
 
By Sub-claim~\ref{SC_zeta}, we find that 
$\hat\rho_r(k,i)=\xi e^{-\eps i-(r-2)k-h_r(k)}$ 
for all $i,k$ such that 
$k/\log^2k \le i\le (\eps/r)^2k$, 
completing the  induction, and thus giving 
Claim~\ref{Cl_hat:rhoind}.
\end{proof}

\subsection{Proof of Lemma~\ref{L_mueps}} 
\label{A_L_mueps}

\begin{proof}[Proof of Lemma~\ref{L_mueps}]
Put $\alpha_{r,\eps}=(1+\eps)\alpha_r$. 
Let $\beta_{r}=\beta_{r}(\alpha_{r,\eps})$ 
and $\beta_*=\beta_*(\alpha_{r,\eps})$. 
For $\beta>0$ and $\gamma\in[0,1)$, let 
$\mu_{r,\eps}(\beta,\gamma)
=\mu_{r,\eps}(\alpha_{r,\eps},\beta,\gamma)$
and $\mu_r^*(\beta)=\mu_r^*(\alpha_{r,\eps},\beta)$.
Let 
$\gamma_{r,\eps}^*(\beta)$
denote the maximizer of 
$\mu_{r,\eps}(\beta,\gamma)$ over $\gamma\in[0,1)$,
which is well-defined, since for all $\gamma\in(0,1)$, 
\begin{equation}\label{E_mureps:cc}
\frac{\partial^2}{\partial\gamma^2}\mu_{r,\eps}(\beta,\gamma)
-\frac{\beta}{(1-\gamma)^2}
-\frac{\alpha_{r,\eps}\beta^r}{(r-2)!}(1-\gamma)^{r-2}
<0
\end{equation}
and 
$\lim_{\gamma\to1^-}\mu_{r,\eps}(\beta,\gamma)=-\infty$.
Finally, put 
$\gamma_{r,\eps}(\beta)
=
\min\{\gamma_{r,\eps}^*(\beta),(\eps/r)^2\}$. 

We show that 
$\mu_{r,\eps}(\beta,\gamma_{r,\eps}(\beta))$
is bounded away from 0 for $\beta\in[\beta_r^*,\beta_r^*+\delta]$, 
for some $\delta>0$. 
By Lemma~\ref{L_hat:EkiLB}, the result follows. 

\begin{claim}\label{Cl_gam*}
For $\gamma\in(0,1)$, let
\[
\beta_{r,\eps}(\gamma)
=\frac{(1/(1-\gamma)+\eps)^{1/(r-1)}}{1-\gamma}\beta_r
\]
and put 
\[
\beta_{r,\eps}=
\lim_{\gamma\to0^+}\beta_{r,\eps}(\gamma)=
(1+\eps)^{1/(r-1)}\beta_r.
\]
We have that 
\begin{enumerate}[nolistsep,label={(\roman*)}]
\item $\gamma_{r,\eps}^*(\beta)=0$, for all $\beta\le\beta_{r,\eps}$, 
\item for $\beta>\beta_{r,\eps}$, $\gamma=\gamma_{r,\eps}^*(\beta)$
if and only if $\beta=\beta_{r,\eps}(\gamma)$, and 
\item  $\gamma_{r,\eps}^*(\beta)$ is increasing in $\beta$, 
for $\beta\ge \beta_{r,\eps}$. 
\end{enumerate}
\end{claim}

\begin{proof}
By \eqref{E_mureps:cc},
we have that $\mu_{r,\eps}(\beta,\gamma)$ is concave in $\gamma$. 
Therefore, since 
\[
\frac{\partial}{\partial\gamma}\mu_{r,\eps}(\beta,\gamma)
-\beta\left(
\frac{1}{1-\gamma}+\eps-\frac{\alpha_{r,\eps}\beta^{r-1}}{(r-1)!}(1-\gamma)^{r-1}
\right)
\]
and hence, for any $\xi>0$, 
\[
\frac{\partial}{\partial\gamma}\mu_{r,\eps}(\xi\beta_r,\gamma)
=
-\xi\beta_r\left(
\frac{1}{1-\gamma}+\eps-\xi^{r-1}(1-\gamma)^{r-1}
\right), 
\]
the first two claims follow. 
The third claim is a consequence of the second claim and 
the fact that 
$\beta_{r,\eps}(\gamma)$ is increasing in $\gamma$.  
\end{proof}

By the following claims, we obtain the lemma
(as we discuss below the statements).

\begin{claim}\label{Cl_omega}
 For $\beta>0$ and $\gamma\in[0,1)$, let 
\[
\omega_{r,\eps}(\beta,\gamma)
=
\mu_{r,\eps}(\beta,\gamma)
-
\mu_{r}^*(\beta).
\]
We have that 
\begin{enumerate}[nolistsep,label={(\roman*)}]
\item $\omega_{r,\eps}(\beta,\gamma_{r,\eps}(\beta))=0$, 
for all $\beta\le\beta_{r,\eps}$, and
\item $\omega_{r,\eps}(\beta,\gamma_{r,\eps}(\beta))$ 
is increasing in $\beta$, for 
$\beta\ge\beta_{r,\eps}$.
\end{enumerate}
\end{claim}

\begin{claim}\label{Cl_beta}
We have that $\beta_{r,\eps}<\beta_*$. 
\end{claim}

Indeed, the claims
together imply that 
 $\omega_{r,\eps}(\beta_*,\gamma_{r,\eps}(\beta_*))>0$.
Therefore, since $\mu_r^*(\beta_*)=0$, we thus have that 
$\mu_{r,\eps}(\beta_*,\gamma_{r,\eps}(\beta_*))>0$. 
Therefore, by the continuity of 
$\mu_{r,\eps}(\beta,\gamma_{r,\eps}(\beta))$  
in $\beta$, it follows that 
$\mu_{r,\eps}(\beta,\gamma_{r,\eps}(\beta))>0$
for all $\beta\in[\beta_*,\beta_*+\delta]$, for some $\delta>0$.
As discussed the lemma 
follows, applying Lemma~\ref{L_hat:EkiLB}.

\begin{proof}[Proof of Claim~\ref{Cl_omega}]
  The first claim follows by \eqref{E_mueps-gam0} 
  and Claim~\ref{Cl_gam*}(i).

For the second claim, we show that 
(a) $\omega_{r,\eps}(\beta,\gamma_{r,\eps}^*(\beta))$ 
is increasing in $\beta$, for 
$\beta\ge\beta_{r,\eps}$ 
such that $\gamma_{r,\eps}^*(\beta)\le (\eps/r)^2$,
and (b) $\omega_{r,\eps}(\beta,(\eps/r)^2)$ is increasing in $\beta$,
for $\beta\ge\beta_{r,\eps}$. 
By Claim~\ref{Cl_gam*}(iii), this implies the claim. 

Since $\gamma_{r,\eps}^*(\beta)$ 
maximizes $\mu_{r,\eps}(\beta,\gamma)$, and so 
$\partial\omega_{r,\eps}(\beta,\gamma_{r,\eps}^*(\beta))/\partial\gamma=0$,
it follows that  
\[
\frac{\partial}{\partial \beta}\omega_{r,\eps}(\beta,\gamma_{r,\eps}^*(\beta))
=
\frac{\partial}{\partial \beta}\omega_{r,\eps}(\beta,\gamma)
\big|_{\gamma=\gamma_{r,\eps}^*(\beta)}.
\] 
Hence, by Claim~\ref{Cl_gam*}(ii), to establish (a)
we show that  for all $\gamma\le (\eps/r)^2$,
$\partial\omega_{r,\eps}(\beta_{r,\eps}(\gamma),\gamma)/\partial \beta
>0$. 
To this end, we observe that 
\begin{equation}\label{E_omega:dbeta}
\frac{\partial}{\partial \beta}\omega_{r,\eps}(\beta,\gamma)
=
\log(1-\gamma)-\eps\gamma
+\frac{\alpha_{r,\eps}\beta^{r-1}}{(r-1)!}(1-(1-\gamma)^r). 
\end{equation}
Setting $\beta=\beta_{r,\eps}(\gamma)$, 
the above expression simplifies 
as
\[
\log(1-\gamma)-\eps\gamma
+\frac{1/(1-\gamma)+\eps}{(1-\gamma)^{r-1}}
(1-(1-\gamma)^r).
\]
By the inequalities $(1-x)^y\le 1/(1+xy)$ and 
$\log(1-x)\ge-x/(1-x)$, this expression 
is bounded from below by 
\[
-\frac{\gamma}{1-\gamma}
-\eps\gamma
+
(1+(r-1)\gamma)
\left(\frac{1}{1-\gamma}+\eps\right)
\left(1-\frac{1}{1+\gamma r}\right)
\]
which factors as 
\[
\frac{\gamma(1+\eps(1-\gamma))}{(1-\gamma)(1+\gamma r)}
(r-1+\gamma r(r-2))>0
\]
and (a) follows. 

Similarly, we note that by \eqref{E_omega:dbeta}, 
for any $\beta\ge\beta_{r,\eps}$
and $\gamma>0$, 
\begin{align*}
\frac{\partial}{\partial \beta}\omega_{r,\eps}(\beta,\gamma)
&\ge
\log(1-\gamma)-\eps\gamma
+\frac{\alpha_{r,\eps}\beta_{r,\eps}^{r-1}}{(r-1)!}(1-(1-\gamma)^r)\\
&=
\log(1-\gamma)-\eps\gamma+(1+\eps)(1-(1-\gamma)^r).
\end{align*}
Hence, using the same 
bounds for $(1-x)^y$ and $\log(1-x)$ as above,
we find that 
for all such $\beta\ge\beta_{r,\eps}$, 
$\partial\omega_{r,\eps}(\beta,(\eps/r)^2)/\partial\beta$ 
is bounded from below by 
\[
\frac{\eps^2
(r^3(r-1)(1+\eps)-2r^2\eps^2-r(2r-1)\eps^3+\eps^5)
}{(r-\eps)(r+\eps)(r+\eps^2)r^2}.
\]
For $\eps<1/r$, the numerator is bounded from below by
\[
\eps^2\left(r^3(r-1)-2-\frac{2r-1}{r^2}\right)
=
\frac{\eps^2}{r}\left(r^6-r^5-2r^2-2r+1\right)>0
\]
since $r\ge2$.
Hence $\partial\omega_{r,\eps}(\beta,(\eps/r)^2)/\partial\beta>0$, giving  (b), 
and thus completing the proof
of the second claim.
\end{proof}

\begin{proof}[Proof of Claim~\ref{Cl_beta}]
By Lemma~\ref{L_beta*}, the claim is equivalent to 
$\mu_{r}^*(\beta_{r,\eps})>0$. 
To see this, we note that 
\[
\beta_r=\left(\frac{(r-1)!}{\alpha_{r,\eps}}\right)^{1/(r-1)}
=\left(\frac{1}{1+\eps}\right)^{1/(r-1)}\left(\frac{r-1}{r}\right)^2,
\]
and hence by \eqref{E_mu*}, for any $\xi>0$, 
we have that 
\begin{align*}
\mu_{r}^*(\xi^{1/(r-1)}\beta_{r})
&=
r-
\xi^{1/(r-1)}\beta_{r}
\left(r-2+\frac{\xi}{r}-\log{\xi}\right)\\
&=
r-
\left(\frac{r}{r-1}\right)^2
\left(\frac{\xi}{1+\eps}\right)^{1/(r-1)}
\left(r-2+\frac{\xi}{r}-\log{\xi}\right).
\end{align*}
In particular,  
\[
\mu_{r}^*(\beta_{r,\eps})
=
r-
\left(\frac{r}{r-1}\right)^2
\left(r-2+\frac{1+\eps}{r}-\log(1+\eps)\right).
\]
Therefore, by the bound $\log(1+x)>x/(1+x)$, 
we find that 
\[
\mu_{r}^*(\beta_{r,\eps})
>
\frac{\eps r(r-1-\eps)}{(1+\eps)(r-1)^2}>0
\]
as required.
\end{proof}

As discussed, Lemma~\ref{L_mueps} follows by 
Claim~\ref{Cl_omega},\ref{Cl_beta}.
\end{proof}

\subsection*{Acnowledgments}

The authors would like to thank the Isaac Newton Institute for
Mathematical Sciences, Cambridge, and the organizers of the program
\emph{Random Geometry}, supported by EPSRC Grant Number EP/K032208/1,
during which progress on this project was made.  OA was further
supported by the Simons foundation and NSERC.  BK was supported by NSERC
and Killam Trusts, and attended the above program with the support of
an NSERC Michael Smith Foreign Study Supplement.

\providecommand{\bysame}{\leavevmode\hbox to3em{\hrulefill}\thinspace}
\providecommand{\MR}{\relax\ifhmode\unskip\space\fi MR }
\providecommand{\MRhref}[2]{%
  \href{http://www.ams.org/mathscinet-getitem?mr=#1}{#2}
}
\providecommand{\href}[2]{#2}

\end{document}